\documentclass[a4paper,twoside,11pt]{article}
\usepackage[english]{babel}
\usepackage[T1]{fontenc}
\usepackage[ansinew]{inputenc}
\usepackage{geometry}
\usepackage{color} 
\usepackage{lmodern}

\usepackage{graphicx}

\usepackage{amsmath} \numberwithin{equation}{section}
\usepackage{amsthm}
\usepackage{amsfonts}
\usepackage{amssymb}
\usepackage{graphics}
\usepackage{cases}

\theoremstyle{plain}
\newtheorem{theorem}{Theorem}[section]
\newtheorem{lemma}[theorem]{Lemma}
\newtheorem{corollary}[theorem]{Corollary}
 
\newtheorem{remark}[theorem]{Remark}
\theoremstyle{definition}
\newtheorem{definition}[theorem]{Definition}
    
\newtheorem{example}[theorem]{Example}
\usepackage{color}
\definecolor{orange}{rgb}{1,0.5,0}

\usepackage{enumerate}
\newcommand\C{\mathbb C}         % K\UTF{FFFD}rper der komplexen Zahlen
\newcommand\R{\mathbb R}         % K\UTF{FFFD}rper der reellen Zahlen
\newcommand\Z{\mathbb Z}         % Ring der ganzen Zahlen
\newcommand\N{\mathbb N}         % Ring der ganzen Zahlen
\newcommand\Ha{\mathbb H}       
\newcommand\D{\mathbb D}  
\newcommand\T{\mathbb T}  
\newcommand{\meu}{\alpha} 
\newcommand{\til}{\mu} 
\renewcommand\Im{\mathsf{Im}}        
\renewcommand\Re{\mathsf{Re}}       
\newcommand{\cP}{{\rm\bf P}}  % Probability measures

\newcommand{\supp}{\operatorname{supp}} 

\usepackage[pdfstartview=FitH]{hyperref}

\allowdisplaybreaks[1]

\begin{document}
\parindent 0pt 

\setcounter{section}{0}

\title{Limits of radial multiple SLE and a Burgers-Loewner differential equation}
\date{\today}
\author{Ikkei Hotta\footnote{The first author was supported by JSPS KAKENHI Grant no. 17K14205.} \and Sebastian Schlei{\ss}inger\footnote{The second author was supported by the German Research Foundation (DFG), project no. 401281084.}}

\maketitle

\begin{abstract}
We consider multiple radial SLE as the number of 
curves tends to infinity. We give conditions that imply the tightness of the associated processes given by the Loewner equation.
In the case of equal weights, the infinite-slit limit is described by a Loewner equation whose Herglotz vector field is given by a Burgers 
differential equation.\\
 Furthermore, we investigate a more general form of the Burgers equation. On the one hand, it appears in connection 
with semigroups of probability measures on $\T$ with respect to free convolution. On the other hand, 
the Burgers equation itself is also a Loewner differential equation for 
certain subordination chains.
\end{abstract}



\tableofcontents

\parindent 0pt

\newpage

\section{Introduction}

The Schramm-Loewner evolution SLE($\kappa$) has been proved to describe the scaling limit of various curves from stochastic geometry 
and statistical physics, see \cite{Lawler:2005}. There are several approaches to generalize this framework to ``multiple SLE'', which 
concerns the simultaneous distribution of $N \in \N$ SLE curves within a given domain, see \cite{Karrila} and the references therein for the historical development and the recent progress.  
Also multiple SLE can be motivated by scaling limits of curves from well known random models. One might think of the critical Ising model which produces several interface curves between clusters of spin -1 and +1, see \cite{Koz09}.\\

In the present article, we consider a multiple radial SLE which consists of $N$ non-intersecting simple SLE($\kappa$) curves ($\kappa\in[0,4]$) connecting $N$ points $x_1,...,x_N$ on the unit circle $\T$ to $0$ within the unit disc $\D$. These curves can be generated by a version of Loewner's differential equation, where we also need to specify weights $\lambda_1,...,\lambda_N \in (0,1)$ with $\lambda_1+...+\lambda_N=1$. The weight $\lambda_k$ can be thought of as the speed of the growth of the $k$-th curve.\\

We address the following question: Which conditions on the points $x_1,...,x_N$ and the weights $\lambda_1,...,\lambda_N$ imply the existence of a limit growth process as $N$ goes to infinity?\\

The following pictures show numerical simulations of such curves.

\begin{figure}[h]
\rule{0pt}{0pt}
\centering
\includegraphics[width=12cm]{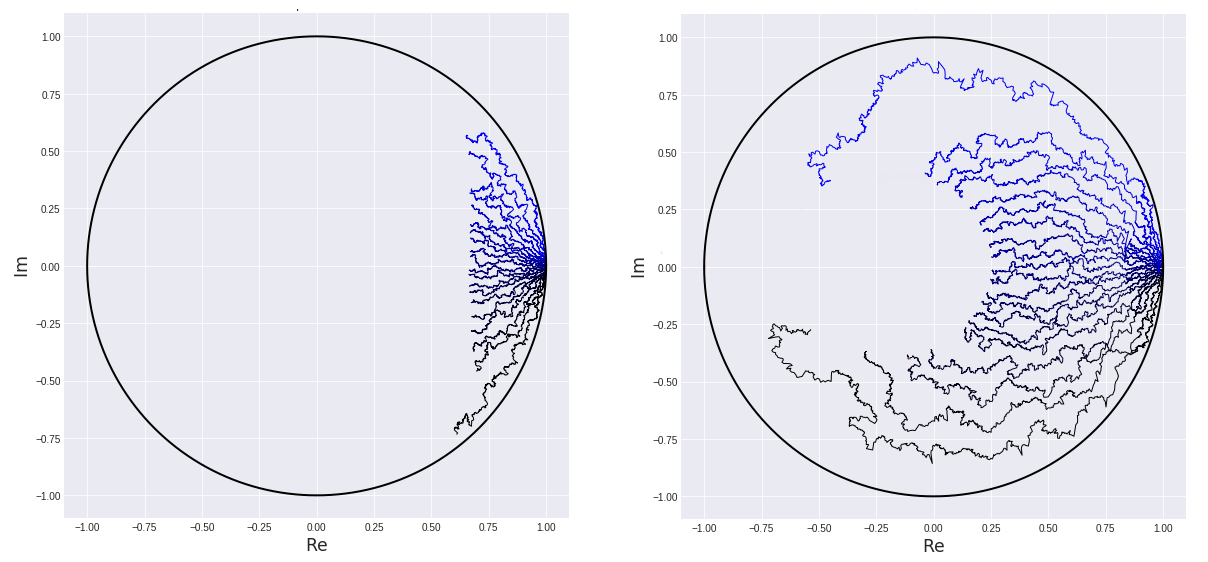}
\caption{$N=20$ multiple SLE(2) curves with equal weights, i.e. $\lambda_k=1/N$, starting near the point $1$ at time $t=0.1$ (left) and $t=1$ (right).}
\end{figure}
	
\begin{figure}[h]
\rule{0pt}{0pt}
\centering
\includegraphics[width=12cm]{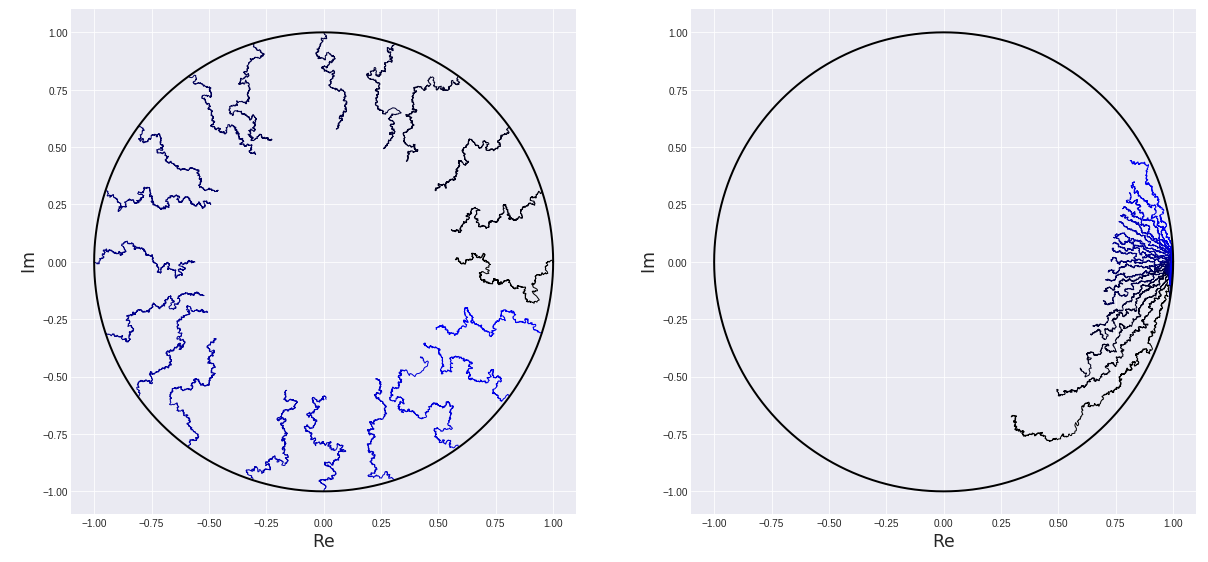}
\caption{Left: $N=20$ multiple SLE(2) curves with equal weights and equally spaced initial points. Right: $N=20$ multiple SLE(2) curves 
starting near the point $1$ with weights $\lambda_1=1/s, \lambda_2=1/(2s),...,\lambda_N=1/(Ns)$, where $s=\sum_{k=1}^N \frac1{k}$.}
\end{figure}	
	
\newpage 

We find conditions on the points and the weights that guarantee tightness of the associated stochastic processes (Section \ref{sec_tightness}). If all the weights are equal ($\lambda_k=1/N$), 
we obtain a simple limit equation involving the a certain complex Burgers equation (Section \ref{Sec_Sim}).\\

These results can be seen as the radial analogue of the works \cite{delMonacoSchleissinger:2016, MR3764710, HK18}, which considered 
this question for chordal multiple SLE. The assumptions in our main result (Theorem \ref{thm:1}) are basically the same as in \cite[Theorem 2.5]{MR3764710}. 
One difference between the radial and the chordal case is the compactness of $\T$, and thus the compactness of the space of all probability measures on $\T$ endowed with the topology induced by weak convergence,
versus the non-compactness of $\R$. The chordal case requires an additional tightness assumption for the starting points on $\R$ (condition (c) in \cite[Theorem 2.5]{MR3764710}).\\

In addition to these works, we point out that the Burgers equation for the unit disc is an instance of a more general evolution equation 
for free semigroups of probability measures on the unit circle (Section \ref{sec_burg_loew}). 
From this perspective, the Burgers equation describes the distributions of a free unitary Brownian motion.\\
We can use the probabilistic interpretation to obtain a geometric property of the hulls described by the Burgers equation, see 
Theorem \ref{thm_geo}.\\
As the literature more often focuses on the chordal case, we also explain the chordal analogue of the free semigroup equation in the appendix.\\

Our paper is organized as follows: In Section \ref{sec_Loewner} we briefly review 
the notion of subordination chains and Loewner's differential equation. In Section \ref{sec_radial}, we investigate the limit behaviour of radial multiple SLE 
as the number of curves converges to infinity. Finally, in Section \ref{sec_burg_loew}, we identify the Burgers equation as an example of a more general differential equation 
arising naturally in free probability theory.

\newpage

\section{Subordination chains and the Loewner equation}\label{sec_Loewner}
We let $\D=\{z\in\C \,|\, |z|<1\}$ be the unit disc, $\T=\partial\D$ the unit circle, 
and $RH=\{z\in\C\,|\, \Re(z)>0\}$ the right half plane.

\begin{definition}
Let $\{f_t\}_{t\geq 0}$ be a family of holomorphic functions on 
$\D$ with $f_t(0)=0$. Then,
\begin{itemize}
 \item $\{f_t\}_{t\geq0}$ is called a \emph{(decreasing and normalized)
Loewner chain} if every $f_t$ is univalent, $f_t(\D)\subset f_s(\D)$ whenever $0\leq s\leq t$, 
$f_0(z)\equiv z$, and $f'_t(0)=e^{-ta}$ for some $a\in RH$.
\item The family $\{f_t\}_{t\geq0}$ is called a \emph{(decreasing and normalized)
subordination chain} if $f_t = f_0 \circ g_t$ for all $t\geq 0$, 
where $\{g_t\}_{t\geq0}$ is a (decreasing and normalized) 
Loewner chain.
\end{itemize}
\end{definition}

Usually, the literature focuses on increasing Loewner chains where 
$f_s(\D) \subset f_t(\D)$ whenever $0\leq s\leq t$. Clearly, 
if $\{f_t\}_{t\in [0,T]}$ is an increasing Loewner chain, then 
$\{f_{T-t}\}_{t\in [0,T]}$
is (a part of) a decreasing Loewner chain.\\
% In case of a decreasing Loewner chain we have $f_t(\D)\subset f_s(\D)$ whenever $0\leq s\leq t$ and 
% $w_{s,t}$ is given by $w_{s,t}=f_s^{-1}\circ f_t$.

Loewner chains have been introduced by C. Loewner in \cite{MR1512136} in order to attack the Bieberbach conjecture, a coefficient problem 
for univalent functions. We refer the reader to \cite[Chapter 6]{P75} 
and \cite[Chapter 3]{duren83} for the classical treatment of Loewner chains, and to \cite{AbateBracci:2010} for a summary of more recent developments.

Increasing subordination chains (with $f_t'(0)\not=0$) have been considered 
by Pommerenke in \cite[Satz 4]{Pom:1965}.
The case of decreasing Loewner chains and subordination chains is 
handled in \cite{CDMG14}. Subordination chains of the above form are related to Loewner's 
partial differential equation. In our situation we have the following relationship.

\begin{theorem}\label{Pom_sub}
Let $\{F_t\}_{t\geq0}$ be a subordination chain. Then there exists 
a Herglotz vector field $h:\D\times [0,\infty) \to \C$, i.e. 
$z\mapsto h(z,t)$ is holomorphic with $h(0,t)=a\in RH$, $\Re (h(z,t))>0$ for all 
$z\in \D$ and all $t\geq0$ and $t\mapsto h(z,t)$ is measurable for all $z\in\D$, 
such that $\{F_t\}_{t\geq 0}$ is a solution of the initial value problem
\begin{equation}\label{Loe_PDE}
 \frac{\partial}{\partial t}f_t(z) = -z h(z,t) \cdot \frac{\partial}{\partial z}f_t(z) \quad 
 \text{for almost every $t\geq 0$, \quad $f_0= F_0$.}
\end{equation}
Conversely, if $h$ is a Herglotz vector field, then \eqref{Loe_PDE} has a unique 
solution $\{f_t\}_{t\geq0}$, which is a subordination chain.
\end{theorem}
\begin{proof}Let ${F_t}$ be a subordination chain with $F_t=F_0\circ g_t$ for all $t\geq 0$, 
where $\{g_t\}$ is a Loewner chain. Note that if $F_0$ is constant (i.e., $F_0(z)=0$ for all $z$), then 
$F_t = F_0$ for all $t\geq 0$ and $F_t$ satisfies \eqref{Loe_PDE} for any Herglotz vector field. \\
Now it is known that $g_t$ satisfies \eqref{Loe_PDE} with 
initial value $f_0(z)\equiv z$, see \cite[Theorem 3.2]{CDMG14}.
The form of the right side and the normalization $h(0,t)=a$ 
follows from the normalization $g_t(0)=0$, $g'_t(0)=e^{-ta}$. A simple computation shows that
$\{F_t\}_{t\geq 0}$ satisfies \eqref{Loe_PDE}.\\
The converse statement follows from \cite[Theorem 3.2]{CDMG14}, \cite[Theorem 1.11]{CDMG14}.
\end{proof}

\begin{remark}
 Assume that $F_0(z)\equiv z$. Then the solution to \eqref{Loe_PDE} is a Loewner chain. 
 Note that two different Herglotz vector fields $h_1$ and $h_2$ can yield the same 
Loewner chain, but if they do, they must coincide for almost every $t\geq 0$; see the comment on ``essential uniqueness'' in \cite[Theorem 3.2]{CDMG14}.\\
Thus, Loewner chains are in a one-to-one correspondence with Herglotz vector fields modulo changes on 
null subsets of the time domain $[0,\infty)$.\\
Also, note that the constant subordination chain $f_t(z)=0$ for all $z\in\D$, $t\geq0$, trivially 
satisfies \eqref{Loe_PDE} for every Herglotz vector field.
 \end{remark}

%(cite also  \cite{Pom:1965}, \cite{Becker:1972} or \cite{Becker:1976}), \cite{Yanagihara} ?)

Now consider a decreasing Loewner chain $\{f_t\}_{t\in [0,T]}$.
Then $t\mapsto g_t(z):=f_t^{-1}(z)$ satisfies the following ordinary differential equation
(see again \cite[Theorem 3.2]{CDMG14})
\begin{equation}\label{Loe_DE}
 \frac{\partial}{\partial t}g_t(z) = g_t(z) h(g_t(z),t)\quad 
 \text{for almost every $t\in [0,T]$ and $z\in f_T(\D)$.}
\end{equation}

In general, we can solve this equation also for $z\in f_0(\D)$, 
but the solution might not be defined for all $t\in[0,T]$ then.  \\

In some cases, it might also be useful to regard the time-reverse version of \eqref{Loe_DE}.
Let $T>0$ and consider

\begin{equation}\label{Loe_DE_rev}
 \frac{\partial}{\partial t}h_t(z) = - h_t(z) h(h_t(z),T-t)\quad 
 \text{for almost every $t\in [0,T]$, $z\in \D$, and $h_0(z)\equiv z$.}
\end{equation}
This equation has a unique solution for every $z\in\D$ and each $h_t:\D\to\D$ is a univalent function. We have $h_T = g_T^{-1}$ (but in general,
 $h_t\not= g_t^{-1}$). In the literature, $(h_t)$ is also called the associated inverse, reverse, or backward Loewner flow, see, e.g., \cite[p. 95]{Lawler:2005}.
 % A detailed study of the relation between equation \eqref{Loe_DE_rev} and \eqref{Loe_PDE}, \eqref{Loe_DE} can be found in \cite{CDMG14}.

\section{Infinite-slit limits of radial multiple SLE}\label{sec_radial}

In the literature the generalization of SLE to multiple SLE curves is discussed in two approaches, a global and a local approach. 
While the local approach describes these curves in a framework of partition functions as solutions to a system of partial differential equations, 
the global approach directly constructs the probability space of the curves by means of the Brownian loop measure.\\
There are several works on the relation between the two approaches and we refer to \cite{Karrila} and the references therein for the recent progress.\\ 

In the simple situation we are looking at, both approaches can be used and we will refer to the work \cite{MR2358649}. Here, we need to restrict ourselves to the case $\kappa\in[0,4]$, which guarantees that the SLE curves are indeed globally defined, see \cite[Section 8]{MR2358649}. 
%One can still regard the corresponding Loewer equation for $\kappa >4$, but the genertaed curves might hit each other.#

\subsection{Radial single SLE}

Fix some  $\kappa\in[0,4]$. The radial Schramm-Loewner evolution SLE($\kappa$) is defined 
as the evolution of a random simple curve $\gamma:[0,\infty)\to \overline{\D}$ with 
$\gamma(0)=1$ and $\gamma(0,\infty)\subset \D\setminus\{0\}$.  This curve is defined as follows:\\
Let $B_t$ be a standard one-dimensional Brownian motion and consider the radial Loewner equation (of the form \eqref{Loe_DE}) 
with driving function 
$e^{i\sqrt{\kappa} B_t}$, i.e. 

\begin{equation}\label{radial_SLE}
\frac{\partial}{\partial t}g_{t}(z) = g_{t}(z) \frac{e^{i\sqrt{\kappa} B_t}+g_{t}(z)}{e^{i\sqrt{\kappa} B_t}-g_{t}(z)},
\quad g_{0}(z)=z\in\D.
\end{equation}
Then $g_t$ is a conformal mapping of the form $g_t:\D\setminus \gamma[0,t]\to\D$, with 
$g_t(0)=0$ and $g_t'(0)=e^t$, for a random simple curve $\gamma$. As $t\to \infty$ we have 
$\gamma(t)\to 0$, i.e. $\gamma$ connects the boundary point $p=1$ to the interior point $q=0\in \D$. 
By conformal invariance, this construction can easily be transferred to any Jordan domain $D$ 
and any $p\in \partial D$, $q\in D$.

\begin{remark}
If we replace the point $1$ by some $p\in\T$, we simply rotate the driving function $e^{i\sqrt{\kappa}B_t}$ 
to $V(t)=e^{i\sqrt{\kappa}B_t}p$. By It\={o}'s  formula, the resulting differential equation can also be written as
\begin{equation}\label{radial1}
\frac{\partial}{\partial t}g_{t}(z) = g_{t}(z) \frac{V(t)+g_{t}(z)}{V(t)-g_{t}(z)},
\quad g_{0}(z)=z\in\D,
\end{equation}
with 
\begin{equation}\label{sde1}
dV(t) = -\frac{\kappa}{2} V(t)dt + i\sqrt{\kappa}V(t)dB_t, \quad V(0)=p\in\T.
\end{equation}
The driving function $V$ is simply the image of the tip of the generated curve $\gamma$ under the mapping $g_t$, i.e. $V(t)=g_t(\gamma(t))$.
\end{remark}

\begin{remark} For $\kappa>4$, equation \eqref{radial_SLE} still yields random conformal mappings 
of the form $g_t:\D\setminus K_t\to \D$, but $K_t$ is not a simple curve anymore.
For the Loewner equation and SLE, we refer the reader to \cite{Lawler:2005}. 
\end{remark}

\subsection{Radial SLE($\kappa, \varrho$) }

SLE can be generalized by adding force points, which leads to SLE($\kappa, \varrho$). Again we consider $\kappa\in[0,4]$. Let $m\in\N$ 
and $\varrho_1, ..., \varrho_m\in\R$. Furthermore, let $x_1,...,x_m\in \T$,  
$p\in \T\setminus\{x_1,...,x_m\}$, and let $B_t$ be a Brownian motion. Then radial SLE($\kappa, \varrho_1,...,\varrho_m$) 
with force points $x_1,...,x_m$ and starting point $p$ is defined via \eqref{radial1} 
and the driving function $V$ is given by the following generalization of \eqref{sde1}:

\begin{eqnarray*}
&&dV(t) = -\frac{\kappa}{2} V(t)dt + 
 \sum_{j=1}^m \frac{\varrho_j}{2} V(t)\frac{X_j(t)+V(t)}{X_j(t)-V(t)} dt+
i\sqrt{\kappa}V(t)dB_t, \quad V(0)=p,\\
&& dX_j(t) = X_j(t) \frac{V(t)+X_j(t)}{V(t)-X_j(t)}dt, \quad X_j(0) = x_j.
\end{eqnarray*}
Note that $X_j(t)$ simply solves the Loewner equation, i.e. $X_j(t) = g_t(x_j)$. The solution 
$g_t$ describes the evolution of a random simple curve (because SLE($\kappa, \varrho$) can be represented as a weighted SLE($\kappa$) $\gamma$ and we chose $\kappa\in[0,4]$, see \cite[Section 5]{SW05}),  
but $g_t$ may not be defined for all $t\geq0$. In the first place, it is only defined as long as $V(t)-X_j(t)\not=0$ for all $j$.

\begin{remark}\label{chordal_version} Radial SLE($\kappa, \varrho$)
can also be considered for force points 
within $\D$. In particular, this more general framework unifies radial and chordal SLE, see  \cite[Theorem 3]{SW05}. This explains the factor $\frac1{2}$ for the weights $\frac{\varrho_j}{2}$ 
in the above SDE.\\
Let $C:\D\to \Ha=\{z\in\C\,|\, \Im(z)>0\},$ $C(z)=i\frac{1+z}{1-z}$, 
be the Cayley transform. The chordal Loewner equation
\begin{equation}
\frac{d}{dt}g_{t}(z) = \frac{2}{g_{t}(z)-U(t)}, \quad g_{0}(z)=z\in\Ha,
\end{equation}
where $U$ is real-valued, yields conformal mappings of the form $g_t:\Ha\setminus K_t\to \Ha$.\\
Assume that $p, x_1,...,x_m \in \T\setminus\{1\}$. By \cite[Theorem 3]{SW05}, the sets $K_t$ correspond (up to a time change) to $C(\gamma[0,t])$ if $U$ satisfies 
 \begin{eqnarray*}
%&&dU(t) = \sqrt{\kappa}dB_t + \sum_{j=1}^m \frac{\varrho_j}{U(t)-Y_j(t)}dt + \frac{A}{2} \left(\frac{1}{U(t)-S(t)}+\frac{1}{U(t)-\overline{S(t)}}  \right), \quad U(0)=C(p)\in\R,\\
&&dU(t) = \sqrt{\kappa}dB_t + \sum_{j=1}^m \frac{\varrho_j}{U(t)-Y_j(t)}dt + (\kappa-6-\sum_{j=1}^m \varrho_j) \Re\left(\frac{1}{U(t)-S(t)}\right)dt, \quad U(0)=C(p)\in\R,\\
&& dY_j(t) = \frac{2}{Y_j(t)-U(t)}dt, \quad Y_j(0) = C(x_j)\in\R,\\
&& dS(t) = \frac{2}{S(t)-U(t)}dt, \quad S(0) = i = C(0).
\end{eqnarray*}
 Note that $Y_j(t)=g_t(C(x_j))$ and $S(t)=g_t(i)$. 
\end{remark}

\subsection{Radial multiple SLE}

Our setting of radial multiple SLE follows the works \cite{MR2004294} and \cite{MR2358649}.\\

Let $p_1,...,p_N$ be $N$ points on $\T$ and fix again $\kappa\in[0,4]$. We now generate $N$ pairwise disjoint simple curves starting at $p_1,...,p_N$ and growing to $0$ via the Loewner equation.\\

First, fix a time interval 
$[0,T]$ and some coefficients $\lambda_1,...,\lambda_N\in(0,1)$ 
with $\lambda_1+...+\lambda_N =1.$ \\
For $t\in[0, \lambda_1 T]$, we use the Loewner equation \eqref{radial1} to generate a curve $\gamma_1$ which is an SLE$(\kappa, 2,...,2)$ 
with starting point $p_1$ and force points $p_2,...,p_N$. Denote the driving function $V$ by $V_1$ and let $V_j(t)=g_t(p_j)$ for $j=2,...,N$. \\
For $t\in[\lambda_1 T, (\lambda_1+\lambda_2)T]$, we generate an SLE$(\kappa, 2,...,2)$ curve with starting point 
$V_2(\lambda_1 T)=g_{\lambda_1 T}(p_2)$ and force points $V_1(\lambda_1 T)=g_{\lambda_1 T}(\gamma_1(\lambda_1 T)), V_3(\lambda_1 T)=g_{\lambda_1 T}(p_3), ..., V_N(\lambda_1 T)$ via \eqref{radial1} 
(and we use a Brownian motion which is independent from the one used for the interval $[0,T]$). Then $g_{(\lambda_1+\lambda_2)T}$ maps $\D\setminus (\Gamma_1 \cup \Gamma_2)$ conformally onto $\D$, where $\Gamma_1$ and $\Gamma_2$ are two disjoint simple curves starting at $p_1$ and $p_2$ respectively. We extend $V_2$ to $[\lambda_1 T,(\lambda_1+\lambda_2)T]$ by letting $V_2$ be the corresponding driving function, and we extend 
$V_1, V_3, ..., V_N$ via $V_j(t)=g_t(p_j)$. In other words, every $V_j$ has the form 
\[ \text{$V_j(t)=g_t(\text{``force point''})$ or $V_j(t)=g_t(\text{``tip of the growing curve''})$}. \]
We extend this construction for the time interval $[0,T]$ and we obtain $N$ disjoint simple curves $\Gamma_1,...,\Gamma_N$ starting from $p_1,...,p_N$. Now we can repeat the procedure for the interval $[T,2T]$, $[2T, 3T]$, etc., and thus the growth of $N$ random curves $\gamma_1, ...,\gamma_N$, connecting $p_1,...,p_N$ to $0$, is defined for all $t\in[0,\infty)$. \\
It has been shown in \cite{MR2358649} that this procedure does in fact not depend on $\lambda_1,...,\lambda_N$, $T$, and the order of $(p_1,...,p_N)$, i.e. 
the distribution of $(\gamma_1, ...,\gamma_N)$ is the same for all possible choices of these parameters (see the case (ii) in \cite[Section 8.2]{MR2358649} and \cite[Theorem 7]{MR2358649}).
This defines \emph{radial multiple SLE($\kappa$)} (from $p_1,...,p_N\in \T$ to $0$).\\

Note that $V_k$ satisfies 
\begin{equation}\label{first_rho} dV_k(t) =  V_k(t) \frac{V_j(t)+V_k(t)}{V_j(t)-V_k(t)}dt\end{equation}
if $j\not= k$ and the $j$-th curve is growing, and 
\begin{equation}\label{second_rho} dV_k(t) = -\frac{\kappa}{2} V_k(t)dt + 
 \sum_{j\not=k} V_k(t)\frac{V_j(t)+V_k(t)}{V_j(t)-V_k(t)} dt+
i\sqrt{\kappa}V_k(t)dB_t,\end{equation}
when the $k$-th curve is growing.\\

In contrast to the above construction, we can also generate the curves simultaneously.
Now we use the coefficients $\lambda_1,...,\lambda_N$ as ``speeds'' in the 
multi-slit Loewner equation 
\begin{equation*}
\frac{d}{dt}g_{t}(z) = g_{t} (z)\sum_{k=1}^N \lambda_k \frac{X_k(t)+g_{t}(z)}{X_k(t)-g_{t}(z)},
\quad g_{0}(z)=z\in\D.
\end{equation*}
One can now derive a commutation relation for the infinitesimal generators for $X_1,...,X_n$ (and use the fact that 
 $s\mapsto \gamma_1[0,t]\cup ... \cup \gamma_j[0,t+s] \cup ... \cup \gamma_N[0,t]$ describes 
an SLE($\kappa, 2,...,2$) process in $\D\setminus (\gamma_1[0,t]\cup...\cup \gamma_N[0,t])$), which leads to 
%By using the fact that $s\mapsto \gamma_1[0,t]\cup ... \cup \gamma_j[0,t+s] \cup ... \cup \gamma_N[0,t]$ describes 
%an SLE($\kappa, 2,...,2$) process in $\D\setminus (\gamma_1[0,t]\cup...\cup \gamma_N[0,t])$, one can now derive
the following differential equations (see \cite[Section 8.2]{MR2358649}):

\begin{eqnarray*}dX_k(t)=
 -\frac{\kappa \lambda_k}{2} X_k(t)dt + 
 \sum_{j\not=k}(\lambda_k+\lambda_j) X_k(t)\frac{X_j(t)+X_k(t)}{X_j(t)-X_k(t)} dt+
i\sqrt{\kappa \lambda_k}X_k(t)dW_k(t),\quad X_k(0)=p_k,
\end{eqnarray*}
where $W_1,...,W_N$ are independent Brownian motions. (See also \cite[Section 5]{graham} for the corresponding calculations in the chordal case.)

\begin{remark}
 
Informally, we can consider the limit $T\to 0$ for $V_k$ in \eqref{first_rho}, \eqref{second_rho}, which leads to a simultaneous growth of the curves, and 
$dX_k$ is obtained as a convex combination of the $N$ terms with weights $\lambda_1,...,\lambda_N$, i.e. we obtain
\begin{align*}dX_k(t)
 &=  -\frac{\kappa }{2} X_k(t)d(\lambda_kt) + 
 \sum_{j\not=k}  X_k(t)\frac{X_j(t)+X_k(t)}{X_j(t)-X_k(t)} d(\lambda_k t)+\\
&\quad+ i\sqrt{\kappa }X_k(t)dB_k(\lambda_k t) + 
\sum_{j\not=k} X_k(t) \frac{X_j(t)+X_k(t)}{X_j(t)-X_k(t)}d(\lambda_j t)\\
&=
 -\frac{\kappa \lambda_k}{2} X_k(t)dt + 
 \sum_{j\not=k}(\lambda_k+\lambda_j) X_k(t)\frac{X_j(t)+X_k(t)}{X_j(t)-X_k(t)} dt+
i\sqrt{\kappa \lambda_k}X_k(t)dW_k(t),
\end{align*}
where $B_1,...,B_N$ are independent Brownian motions and we used that 
$B_k(ct)=\sqrt{c}W_k(t)$ for another standard Brownian motion $W_k$.
\end{remark}

\subsection{Limits of radial multiple SLE}

Let $N\in\N$ and $x_{N,1},...,x_{N,N}$ be $N$ points on $\T$, ordered in counter-clockwise direction such that 
$1$ lies on the circular arc $(x_{N,N}, x_{N,1}]$. Furthermore, choose $\lambda_{N,1},...,\lambda_{N,N}\in(0,1)$ such
that $\sum_{k=1}^N\lambda_{N,k}=1.$\\

We now grow $N$ radial multiple SLE curves simultaneously for these parameters, i.e. 
we define $N$ random processes $V_{N,1},...,V_{N,N}$ on $\T$ as the solution of the SDE system
\begin{equation}\label{driving_general}
 dV_{N,k}(t) = V_{N,k}(t)\left(\sum_{j\not=k}(\lambda_{N,k}+\lambda_{N,j})\frac{V_{N,j}(t)+V_{N,k}(t)}{V_{N,j}(t)-V_{N,k}(t)}
 -\frac{\kappa\lambda_{N,k}}{2}\right)dt+iV_{N,k}\sqrt{\kappa \lambda_{N,k}}dB_{N,k}(t),
\end{equation}
with $V_{N,k}(0)=x_{N,k}$ and $B_{N,1},...,B_{N,N}$ are $N$ independent standard Brownian motions.
%As $\kappa\in[0,4]$,  the solutions $V_{N,1},...,V_{N,N}$ exist for all $t\geq 0$ and they do not collide.\\

The corresponding $N$-slit Loewner equation 
\begin{equation}\label{multi2}
\frac{d}{dt}g_{N,t}(z) = g_{N,t}(z) \sum_{k=1}^N \lambda_{N,k} \frac{V_{N,k}(t)+g_{N,t}(z)}{V_{N,k}(t)-g_{N,t}(z)},
\quad g_{N,0}(z)=z\in\D,
\end{equation}
describes the growth of $N$ multiple SLE curves growing from $x_{N,1},...,x_{N,N}$ to $0$.\\
The function $z\mapsto g_{N,t}(z)$ is a conformal mapping from $\D\setminus 
(\gamma_{N,1}[0,t]\cup ...\cup \gamma_{N,N}[0,t])$ onto $\D$,
where the curves $\gamma_{N,k}:[0,\infty)\to \overline{\D}\setminus\{0\}$ are non-intersecting
simple curves with $\gamma_{N,k}(0)=x_{N,k}$ and $g_{N,t}$ has the normalization $g_{N,t}(0)=0,$ $g_{N,t}'(0)=e^{t}$.
%In \cite{MR2004294}, the authors consider $\theta_{N,k}(t),$ where $V_{N,k}(t)=e^{i\theta_{N,k}(t)}.$\\

The special case $\lambda_{N,k}=\frac1{N}$ for all $k=1,...,N$ (simultaneous growth) leads to 

\begin{equation}\label{driving_simul}
 dV_{N,k}(t) = V_{N,k}(t)\sum_{j\not=k}\frac{2}{N}\frac{V_{N,j}(t)+V_{N,k}(t)}{V_{N,j}(t)-V_{N,k}(t)}dt-
 \frac{\kappa}{2N}V_{N,k}dt+iV_{N,k}\sqrt{\kappa/N}dB_{N,k}(t), \quad V_{N,k}(0)=x_{N,k}.
\end{equation}

\textbf{Problem:}
We are interested in the limit $N\to\infty$ of the growing curves, i.e. the convergence of $\gamma_{N,1}[0,t]\cup ...\cup \gamma_{N,N}[0,t]$ to a set $K_t$, more precisely:\\
 Fix some $t>0.$ Under which conditions does the sequence 
 $\D\setminus (\gamma_{N,1}[0,t]\cup ...\cup \gamma_{N,N}[0,t])$ of domains converge to a 
 (simply connected) domain $\D\setminus K_t$ with respect to kernel convergence? 
 According to Carath\'{e}odory's kernel theorem, this is equivalent to asking for locally uniform convergence of the mappings $g_{N,t}$ to a conformal mapping 
 $g_t:\D\setminus K_t\to\D.$ Also, we would like to be able to describe $g_t$ again by a Loewner equation.\\

In Section \ref{sec_tightness}, we consider the general case 
\eqref{driving_general} and obtain a tightness result under certain assumptions on the coefficients $\lambda_{N,k}$.\\
In Section \ref{Sec_Sim}, we consider the case \eqref{driving_simul} and we prove convergence of $g_{N,t}$. 
The processes \eqref{driving_simul} are related to 
 Dyson's model for unitary random matrices, whose limit behaviour is investigated in \cite{MR1875671}.

\subsection{Tightness}\label{sec_tightness}

\begin{definition}
\label{definition3.5}
Fix $T>0$ and let $\cP(\T)$ be the space of probability measures on $\T$ endowed with 
the topology of weak convergence. Note that $\cP(\T)$ is a metric space due to the well-known L\'{e}vy-Prokhorov metric. We denote by 
$\mathcal{M}(T)=C([0,T], \cP(\T))$ the space of all continuous measure-valued processes on
$[0,T]$ endowed with the topology of uniform convergence.
\end{definition}

Let $\delta_x$ be the Dirac mass at $x$ and let $\meu_{N,t}=\sum_{k=1}^N \lambda_{N,k} \delta_{V_{N,k}(t)}$. Then equation \eqref{multi2} can be written as 
\begin{equation}\label{Poma}\frac{d}{dt}g_{N,t}=g_{N,t}\int_{\partial \D}\frac{x+g_{N,t}}{x-g_{N,t}}\, \meu_{N,t}(dx).\end{equation}

For every $N\in\N,$ $\meu_{N,t}$ can be regarded as a random element from $\mathcal{M}(T)$. 
We now make three assumptions that allow us to prove tightness of the sequence $\{\meu_{N,t}\}_N$. \\

The first assumption: There exists $C>0$ such that for every $N\in\N:$\begin{equation}\label{max}\tag{a}
\max_{k\in\{1,...,N\}} \lambda_{N,k} \leq \frac{C}{N}.
 \end{equation}

Next, let $L_N:[0,1]\to[0,1]$ be defined by 
$L_N(k/N) = \sum_{j=1}^k \lambda_{N,j}$ for $k=0,...,N$ and define $L_N(x)$ by 
linear interpolation for all $x\in[0,1].$
The family $\{L_N\}_{N\in\N}$ is uniformly bounded by $1$ and equicontinuous by \eqref{max}.
The Arzel\`{a}--Ascoli theorem implies that it is precompact. We will assume that the limit exists:
\begin{equation}\label{imp:2}
L_N(x) \to L(x) \quad \text{uniformly on $[0,1]$ as $N\to\infty.$}\tag{b}
 \end{equation}

While the conditions \eqref{max} and \eqref{imp:2} define a ``convergence of the weights'', 
we also need a precise meaning of ``the convergence of the initial points''. For this purpose, we introduce the ``empirical distribution'' $\til_{N,t}=\sum_{k=1}^N \frac{1}{N} \delta_{V_{N,k}(t)}.$
Our third assumption is 
 \begin{equation}\label{assum2}
\til_{N,0} \longrightarrow \mu \quad \text{as $N\to\infty$} \tag{c}
\end{equation}
for some probability measure $\mu$ on $\T$. \\

For $x\in\T$, let $[1,x]$ be the circular 
arc on $\T$ from $1$ to $x$ in counter-clockwise direction. Let 
$F_{N,t}(x)=\meu_{N,t}([1,x])$ and $G_{N,t}(x)=\til_{N,t}([1,x])$ be the cumulative distribution 
functions. Then, $G_{N,t}$ and $F_{N,t}$ are related as follows: 
\begin{equation}
\label{imp:1}
F_{N,t}(x) = L_N(G_{N,t}(x)).
\end{equation}
Due to \eqref{imp:2}, assumption \eqref{assum2} also implies the convergence 
of $\alpha_{N,0}$.\\

We let $C^2(\T,\C)$ be the space of all twice continuously differentiable functions $f:\T\to \C$.

\begin{theorem}\label{thm:1}
Assume that \eqref{max}, \eqref{imp:2}, and \eqref{assum2} hold. Let $T>0$. Then the sequences $\{\til_{N,t}\}_N$ and $\{\meu_{N,t}\}_N$ are tight with respect to $\mathcal{M}(T)$.\\
 Furthermore, let $\{\til_{N_k,t}\}_k$ be a converging subsequence with limit $\til_t.$ Then $\{\meu_{N_k,t}\}_k$ converges to $\meu_t$ which is defined by \begin{equation}\label{TheQueen}
F_t(x)=L\circ G_t(x),\end{equation} where $F_t$ and $G_t$ are the cumulative distribution functions of the measures $\meu_t$ and 
$\til_t$ respectively. Finally, $\til_t$ satisfies the (distributional) differential equation
\begin{equation}\label{mck:1}  \frac{d}{dt}\left(\int_\T f(x)\, \til_t(dx)\right) =  
-\int_{}\int_{\T^2} \frac{xf'(x) - yf'(y)}{x-y}(x+y) \,
\til_t(dx) \meu_t(dy), \quad \til_0=\til, \end{equation} 
for all $f\in C^2(\T, \C).$
\end{theorem}

%The proof follows the ideas from  \cite{MR1176727, MR1217451, MR1440140}.
%radial case: 
The proof is similar to the one of \cite[Theorem 4.1]{MR1875671} (diffusion processes on $\T$) and \cite[Theorem 2.5]{MR3764710} (chordal multiple SLE).
\begin{proof} 
Proving tightness of the measure-valued processes $\{\til_{N,t}\}_N$ can be reduced to proving tightness of stochastic
\emph{complex-valued} processes, see \cite[Section 3]{MR1217451} and also \cite[Section 1.3]{MR968996}:\\
 
The sequence $\{\til_{N,t}\}_N$ is tight if 
$$ \left\{\int_{\T} f(x) \til_{N,t}(dx)\right\}_{N\in\N}$$ are tight sequences
(with respect to the space space $C([0,T], \C)$ with uniform convergence) for all 
$f\in C^2(\T,\C)$.\\ 
First, let $f\in C^2(\T,\C)$. It\={o}'s formula gives
\begin{align*}
 & d\left(\int_{\T} f(x) \, \til_{N,t}(dx)\right) = 
 d\left(\sum_{k=1}^N \frac{1}{N} f(V_{N,k}(t)) \right) \\ 
 &=  \sum_{k=1}^N 
 \frac{V_{N,k}(t)}{N}\left( f'(V_{N,k}(t)) \left( \sum_{j\not=k}(\lambda_{N,k}+\lambda_{N,j})
 \frac{V_{N,j}(t)+V_{N,k}(t)}{V_{N,j}(t)-V_{N,k}(t)}-\frac{\kappa\lambda_{N,k}}{2} \right) - 
 \frac{f''(V_{N,k}(t)) \kappa \lambda_{N,k}}{2}\right)dt \\
 &\qquad+\sum_{k=1}^N 
 \frac{V_{N,k}(t)}{N} f'(V_{N,k}(t)) i\sqrt{\kappa \lambda_{N,k}} \,dB_{N,k}(t).
 \end{align*} 
 Now we write 
 \begin{align*}
&\sum_{k=1}^N  \frac{V_{N,k}(t)}{N} f'(V_{N,k}(t))\sum_{j\not=k}\lambda_{N,j}
 \frac{V_{N,j}(t)+V_{N,k}(t)}{V_{N,j}(t)-V_{N,k}(t)} dt + \sum_{k=1}^N  \frac{V_{N,k}(t)}{N} f'(V_{N,k}(t))\sum_{j\not=k}\lambda_{N,k}
 \frac{V_{N,j}(t)+V_{N,k}(t)}{V_{N,j}(t)-V_{N,k}(t)} dt \\
&\quad= \iint_{x\not=y} xf'(x)\frac{y+x}{y-x} \,
\til_{N,t}(dx) \meu_{N,t}(dy) \,dt+ \iint_{x\not=y}  xf'(x)\frac{y+x}{y-x}  \,
\meu_{N,t}(dx) \til_{N,t}(dy) \, dt  \\
&\quad= -\iint_{x\not=y} \frac{xf'(x) - yf'(y)}{x-y}(x+y) \,
\til_{N,t}(dx) \meu_{N,t}(dy)  \,dt \\
&\quad=-\iint_{\T^2} \frac{xf'(x) - yf'(y)}{x-y}(x+y) \,
\til_{N,t}(dx) \meu_{N,t}(dy)  \,dt \\
&\quad\qquad - \sum_{k=1}^N \frac{2\lambda_{N,k}V_{N,k}(t)}{N} (f'(V_{N,k}(t))+V_{N,k}(t)f''(V_{N,k}(t))). 
\end{align*} 
Hence we arrive at
 \begin{align*}
 & d\left(\int_{\T} f(x) \, \til_{N,t}(dx)\right) = \\
 &=\,-\underbrace{\iint_{\T^2} \frac{xf'(x) - yf'(y)}{x-y}(x+y) \,
\til_{N,t}(dx) \meu_{N,t}(dy)}_{:=A_N(t)}  \,dt\\
&\qquad- 
\underbrace{\sum_{k=1}^N \frac{2\lambda_{N,k}V_{N,k}(t)}{N} (f'(V_{N,k}(t))+V_{N,k}(t)f''(V_{N,k}(t)))}_{=:B_N(t)} \,
\,dt \\
&\qquad-  \underbrace{\sum_{k=1}^N 
 \frac{V_{N,k}(t)\lambda_{N,k}\kappa}{2N}\left( f'(V_{N,k}(t)) + f''(V_{N,k}(t)) \right)}_{=:C_N(t)}\,dt + \underbrace{\sum_{k=1}^N 
 \frac{V_{N,k}(t)}{N} f'(V_{N,k}(t)) i\sqrt{\kappa \lambda_{N,k}}}_{=:D_N(t)} \,dB_{N,k}(t).
\end{align*} 

As $f'$ and $f''$ are bounded and $\lambda_{N,k}\leq C/N$, we see that the sums $B_N(t), C_N(t)$, and 
$D_N(t)$ converge absolutely (and uniformly on $[0,T]$) to $0$ with probability $1$ as $N\to\infty$. \\
Furthermore, as $xf'(x)$ is continuously differentiable, the term $A_N(t)$ is uniformly bounded 
with respect to $N$ and $t$. By the stochastic Arzel\`{a}-Ascoli theorem (\cite[Thm. 7.3]{MR1700749}), we conclude that 
$\left\{\int_\T f(x) \til_{N,t}(dx)\right\}_{N\in\N}$ is tight. Hence, $\{\til_{N,t}\}_N$ is tight 
and each limit process satisfies equation \eqref{mck:1}. \\
Because of relation \eqref{imp:1} and our assumption \eqref{imp:2} it follows that
the subsequence $\{\meu_{N_k,t}\}_k$ converges provided that $\{\til_{N_k,t}\}_k$ converges,
and that relation \eqref{TheQueen} holds for the limit processes.
\end{proof}

\begin{remark}
The assumptions in Theorem \ref{thm:1} are basically the same as in the chordal case \cite[Theorem 2.5]{MR3764710}. 
One difference between the radial and the chordal case is the compactness of $\T$ (and thus the compactness of $\cP(\T)$) 
versus the non-compactness of $\R$. In \cite[Example 2.16]{MR3764710}, the authors describe an example of chordal multiple SLE processes where the 
corresponding family $\{\alpha_{N,t}\}_N$ of measure-valued processes is not tight. Even though $\mu_{N,0}$ and $\alpha_{N,0}$ converge as $N\to\infty$, the rightmost driving function $V_{N,N}(t)$ 
satisfies $V_{N,N}(t)\to\infty$ as $N\to\infty$ for every $t>0$.\\
Question: Is there an example of radial multiple SLE data $\{x_{N,k}\}$ and $\{\lambda_{N,k}\}$ such that the process $\{\meu_{N,t}\}_N$ is not tight?
\end{remark}

Next we will  show that if $\{\til_{N_k,t}\}_k$ is a converging subsequence, then $g_{N_k,t}$ 
converges as $k\to\infty$.\\

Let $\lambda$ be the Lebesgue measure on $[0,T]$ and let $\mathcal{N}(T)$ be the space
of all finite measures on $\T\times [0,T]$ endowed with the topology of weak convergence.
We will need the following control-theoretic result. A proof can be found in
\cite[Proposition 1]{MR2919205} or \cite[Theorem 1.1]{quantum}.

\begin{theorem}\label{control}
Let $\{\beta_{N,t}\}_{N\in\N}$ be a sequence of processes
from $\mathcal{M}(T)$ and assume that 
$\beta_{N,t} \lambda(dt)$ converges to $\beta_t \lambda(dt)$ within the space $\mathcal{N}(T)$. Denote by $h_{N,t}$ the solution 
to the Loewner equation
$$ \frac{\partial}{\partial t} h_{N,t}(z) = h_{N,t}(z)\int_\T \frac{x+h_{N,t}(z)}{x-h_{N,t}(z)} \, \beta_{N,t}(dx), 
\quad h_{N,0}(z)=z\in\D.$$
Then $h_{N,t}$ converges locally uniformly to $h_t$ for every $t\in[0,T]$ as $N\to\infty,$ where $(h_t)_{t\in[0,T]}$ is the unique solution to
$$ \frac{\partial}{\partial t} h_t(z) = h_t(z)\int_\T \frac{x+h_t(z)}{x-h_t(z)} \, \beta_t(dx), \quad h_0(z)=z\in\D.$$
\end{theorem}
 
Let $\mathcal{C}$ be the set of all $M(z)=\int_\T \frac{x+z}{x-z} \beta(dx)$, where $\beta$
is a probability measure. The measure $\beta$ can be recovered from $M$ by a version of the 
Stieltjes-Perron inversion formula.
%, see \cite[Theorem F.2]{MR2953553}.
Denote the distribution function of $\beta$ by $F(x)$. Then $L\circ F(x)$ is also a distribution function, 
which corresponds to a measure $\hat{\beta}.$ In this way, we obtain a map $\mathcal{L}:\mathcal{C}\to \mathcal{C},$ 
$$\int_\T \frac{x+z}{x-z} \beta(dx) \mapsto \int_\T \frac{x+z}{x-z} \hat{\beta}(dx).$$
The limit of the Loewner equation can now be described as follows.
\begin{corollary}\label{thm:2}
Assume that \eqref{max}, \eqref{imp:2}, and \eqref{assum2} hold. Let $\{\til_{N_k,t}\}_k$ be a converging subsequence with limit $\mu_t$. Then, as $k\to\infty,$ $g_{N_k,t}$ converges in distribution w.r.t. locally uniform convergence to $g_t,$ the solution of the Loewner equation
\begin{equation}\label{Loe:1}
\frac{\partial}{\partial t} g_t = g_t (\mathcal{L} \circ M_t)(g_t),\quad g_0(z)=z\in\D,
\end{equation}
where $M_t=\int_\T \frac{x+z}{x-z}\mu_t(dx)$ solves the (abstract) differential equation 
\begin{equation}\label{abstractBurgers}
\frac{\partial}{\partial t}M_t =  - z M_t  \cdot \frac{\partial}{\partial z}(\mathcal{L} \circ M_t)
- z \frac{\partial}{\partial z}M_t \cdot (\mathcal{L} \circ M_t), \quad M_0(z) = \int_\T \frac{x+z}{x-z} \til(dx).
\end{equation}

\end{corollary}
\begin{proof}
For $z\in\D$ we consider $f(x)=\frac{x+z}{x-z}\in C^2(\T,\C)$. Then 
$\frac{\partial f}{\partial x} = \frac{-2z}{(x-z)^2}$, 
$\frac{\partial f}{\partial z} = \frac{2x}{(x-z)^2}$, and we obtain
\begin{eqnarray*}
&&\frac{xf'(x) - yf'(y)}{x-y}(x+y)=
\frac{x\frac{-2z}{(x-z)^2} - y\frac{-2z}{(y-z)^2}}{x-y}(x+y)=
(x+y)(-2z)\frac{x\frac{1}{(x-z)^2} - y\frac{1}{(y-z)^2}}{x-y}\\
&=& (x+y)(-2z)\frac{x(y^2-2yz+z^2) - y(x^2-2xz+z^2)}{(x-y)(x-z)^2(y-z)^2}\\
&=& (x+y)(-2z)\frac{x(y^2+z^2) - y(x^2+z^2)}{(x-y)(x-z)^2(y-z)^2}\\
&=& (x+y)(-2z) \frac{z^2-xy}{(x-z)^2(y-z)^2}=
-2z\frac{z^2x-x^2y+z^2y-xy^2}{(x-z)^2(y-z)^2}=\\
&=& -2z\frac{y(z^2-x^2)}{(x-z)^2(y-z)^2} + 
-2z\frac{x(z^2-y^2)}{(x-z)^2(y-z)^2}
=z\frac{2y}{(y-z)^2}\frac{x+z}{x-z} + 
z\frac{2x}{(x-z)^2}\frac{y+z}{y-z}.
\end{eqnarray*} 

Let $M_t(z)=\int_\T \frac{x+z}{x-z} \, \til_t(dx)$. 
Then
\begin{eqnarray*}
&&\frac{\partial}{\partial t}M_t(z) = 
-z\int_{}\int_{\T^2} \frac{2y}{(y-z)^2}\frac{x+z}{x-z} + 
\frac{2x}{(x-z)^2}\frac{y+z}{y-z} \,
\til_t(dx) \meu_t(dy)\\ 
&=&- z M_t  \cdot \frac{\partial}{\partial z}(\mathcal{L} \circ M_t)- z \frac{\partial}{\partial z}M_t \cdot (\mathcal{L} \circ M_t).
\end{eqnarray*}

Furthermore, let $g_t$ be the solution to 
$$\frac{d}{dt}g_t = g_t(\mathcal{L}\circ M_t)(g_t), \quad g_0(z)=z\in\D.$$
Fix some $t\geq 0$. 
%The convergence of the process $\til_{N_k,s}$, $s\in[0,t]$, and \eqref{imp:2} imply that also $\meu_{N_k,s}$converges  $k\to\infty$ to a process $\meu_s$ with respect to uniform convergence on the interval $[0,t]$, and the cumulative distribution function $F_t(x)$ of $\meu_t$ is related to $G_t(x)$ by $F_t = L\circ G_t.$ \\
  The canonical mapping $\mathcal{M}(t) \ni \meu_s \mapsto \meu_s\lambda(ds) \in \mathcal{N}(t)$ is continuous. Hence, the Continuous Mapping Theorem (see \cite{MR1700749}, p. 20) implies that $\meu_{N_k,s} \lambda(ds)$ converges in distribution with respect to weak convergence to $\meu_s \lambda(ds)$. \\
 
Hence, Theorem \ref{control} and again the Continuous Mapping Theorem imply that $g_{N_k,t}$, which is the solution to \eqref{Poma},  converges in distribution to $g_t$ with respect to locally uniform convergence.
\end{proof}

\begin{remark}\label{Napoli} The convergence of $\meu_{N,t}$ and $g_{N,t}$ would follow
immediately if we knew that equation \eqref{abstractBurgers}, or equivalently \eqref{mck:1},
had a unique solution. If $\lambda_{N,k}=\frac1{N}$, then \eqref{abstractBurgers} is a standard 
PDE and uniqueness can be shown easily, see Section \ref{Sec_Sim}.
\end{remark}

\subsection{The simultaneous case}\label{Sec_Sim}

\begin{theorem}\label{sim_thm}Let $\lambda_{N,k}=\frac1{N}$ for all $k$ and $N$. 
Assume $\mu_{N,0}$ converges weakly to a probability measure $\mu.$ Then $\mu_{N,t}$ converges to the
deterministic process $\mu_t$ defined as the unique solution to the equation
% \begin{equation}
% \label{eq:1}
% \frac{d}{dt}\left(\int_{\partial\D}f(x) \,\mu_t(dx)\right) = -\int_{}\int_{\partial\D^2} (x+y)\left(x\frac{f'(x)-f'(y)}{x-y} + f'(y)\right) \,
% \mu_{t}(dx) \mu_{t}(dy)
% \end{equation}
\begin{equation}\label{mck:1_sim}  \frac{d}{dt}\left(\int_\T f(x)\, \til_t(dx)\right) =  
-\int_{}\int_{\T^2} \frac{xf'(x) - yf'(y)}{x-y}(x+y) \,
\til_t(dx) \til_t(dy), \quad \til_0=\til, \end{equation} 
for all $f\in C^2(\T,\C).$\\
The conformal mappings $g_{N,t}$ converge locally uniformly to $g_t,$
the solution of the system 
\begin{equation}\label{ord_sim}
\frac{\partial}{\partial t}g_{t}(z) =  g_{t}(z) M_t(g_t(z)),\quad g_{0}(z)=z,
\end{equation}
\begin{equation}\label{sys_sim}
\frac{\partial}{\partial t}M_t(z)= -2z \frac{\partial M_t(z)}{\partial z}M_t(z),\quad 
M_0(z) = \int_\T \frac{x+z}{x-z}\mu(dx).
\end{equation}
\end{theorem}
\begin{proof}
 Clearly, \eqref{mck:1_sim} and \eqref{sys_sim} are the special cases of \eqref{mck:1} and 
 \eqref{abstractBurgers}. It only remains to show that equation \eqref{sys_sim} (and hence also equation 
\eqref{mck:1_sim}) has a unique solution. 
 This is proven in Theorem \ref{at_most} (for an even more general situation).

\end{proof}

Equation \eqref{sys_sim} is in fact the (inviscid) Burgers' equation, which becomes clear after a change of variables.
For $\Im(z)>0$ define $V_t(z)=-iM_t(e^{iz})$. Then $V_t$ maps the upper half-plane into the 
 lower half-plane and we obtain 
\begin{equation}\label{Burg}
 \frac{\partial V_t(z)}{\partial t}= -2 \frac{\partial V_t(z)}{\partial z}V_t(z). 
\end{equation}

\begin{remark}
Let $h_t:\Ha\to\Ha$ be a lift of $g_t(e^{iz}):\Ha\to\D\setminus\{0\}$ with respect to $e^{iz}:\Ha\to\D\setminus\{0\}$, i.e. $g_t(e^{iz})=e^{ih_t(z)}$.
Suppose that $t\mapsto h_t$ is continuous and $h_0(z)\equiv z$. Locally we can write $h_t(z) = -i\log(g_t(e^{iz}))$ and we obtain the differential 
equation
\begin{equation}\label{Burg2}
 \frac{\partial h_t(z)}{\partial t}= V_t(h_t(z)), \quad h_0(z)=z. 
\end{equation}
The system \eqref{Burg}, \eqref{Burg2} is now the same as in \cite[Theorem 1.1]{delMonacoSchleissinger:2016}, 
the only difference being that the 
initial value for $V_0$ is not a Cauchy transform in our setting, but of the form 
\begin{equation}\label{alt}V_0(z)=-i\int_0^{2\pi} \frac{e^{is}+e^{iz}}{e^{is}-e^{iz}}\mu'(ds)\end{equation}
 for a probability measure $\mu'$ on $[0,2\pi]$.
\end{remark}

\begin{example}
 Let $\mu$ be the uniform distribution on $\T$. Then $M_t(z)=1$ solves \eqref{sys_sim} 
and the conformal mapping $g_t$ is given by $g_t(z)=e^{t}z$, 
which maps $e^{-t}\D$ conformally onto $\D$. 
%{\color{blue}Limit distribution: See \cite[Theorem 5.1]{MR1875671}.}
\end{example}

\begin{example}
We can now determine the limit hulls of the curves in Figure 1, where $\lambda_{N,k}=1/N$ and $\mu_{N,0}\to\delta_1$. (The corresponding chordal case is considered in \cite{HK18}.)
A numerical solution of the system \eqref{ord_sim}, \eqref{sys_sim} with $M_0(z)=\frac{1+z}{1-z}$ yields:
\end{example}

\begin{figure}[ht]
\rule{0pt}{0pt}
\centering
\includegraphics[width=12cm]{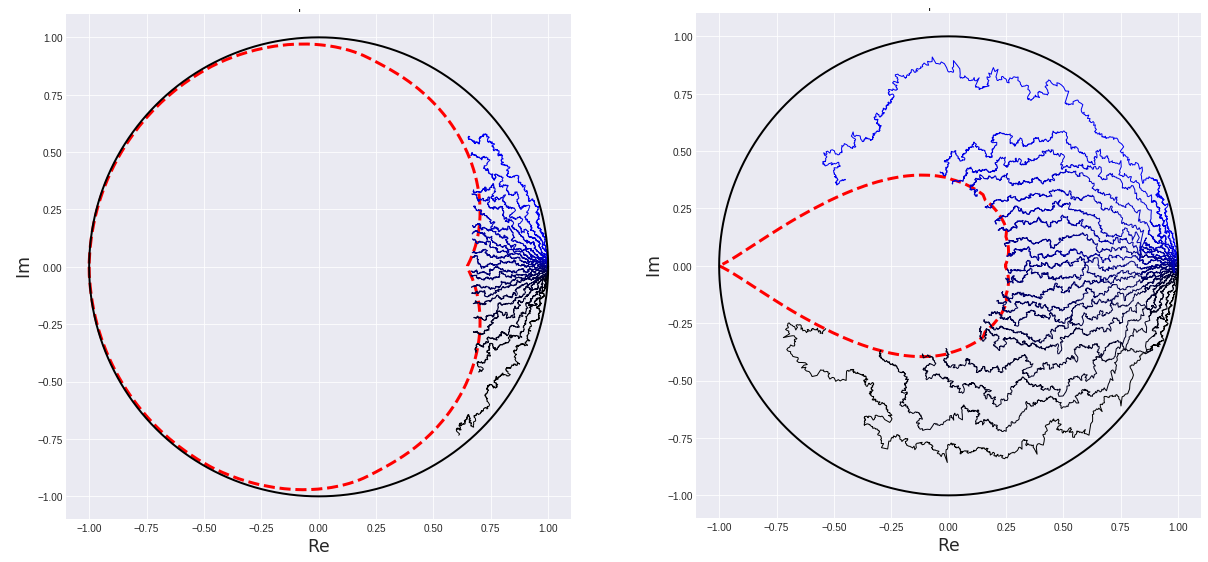}
\caption{The SLE curves from Figure 1 and the boundary of their limit hulls ({\color{red}red}, dashed).}
\end{figure}

Recall that every $M_t$ has the form
\begin{equation*}
M_t(z)=\int_\T \frac{x+z}{x-z}\mu_t(dx).
\end{equation*}

\begin{theorem}\label{thm_geo}
There exists $T>0$ such that $M_t$ and $g_t^{-1}$ both extend analytically to 
$\overline{\D}$, $\Re(M_t(x))>0$ for all $x\in\T$, and $\supp \mu_t=\T$ for all $t\geq T$.
\end{theorem}
\begin{proof}
Due to Theorem \ref{prob_prop}, 
 there exists $T>0$ such that, for all $t\geq T$, $M_t$ extends analytically to $\overline{\D}$ 
with $\Re(M_t(x))>0$ for all $x\in \T$ and $\supp \mu_t=\T$.  \\

Fix any $T_2>T$.
Instead of \eqref{ord_sim}, we consider the time-reverse equation (cf. equation \eqref{Loe_DE_rev})
\begin{equation*}
\frac{\partial}{\partial t}f_{t}(z) =  -f_{t}(z) M_{T_2-t}(f_t(z)),\quad t\in[0,T_2], \quad 
f_{0}(z)=z.
\end{equation*}

Recall that we have $f_{T_2}=g_{T_2}^{-1}$.\\
For $t\in[0,T_2-T]$, the above equation can also be considered on $\T$ and we see that 
$f_{t}$ can be extended analytically to $\overline{\D}$ with $f_t(\T)\subset \D$ for all $t\in (0,T_2-T]$.
Fix $t_0\in(0,T_2-T]$. As $f_{T_2}=\varphi\circ f_{t_0}$ for some holomorphic 
$\varphi:\D\to\D$, we see that $f_{T_2}$ can also be extended analytically to $\overline{\D}$ 
with $f_{T_2}(\T)\subset \D$.

\end{proof}

\begin{remark}
Assume that $\supp \mu_0\not=\T$. 
Then we can compute the smallest time $t$ with $\supp \mu_t=\T$ explicitly.\\
Recall formula \eqref{alt}. Let $x\in \R$ with $x+2k\pi \not \in \supp \mu'$ for all $k\in\Z$. Then $V_0$ can be extended analytically to a neighbourhood of $x$ and
\begin{equation*}
V_0(x) = \int_0^{2\pi} \frac{2\sin(x-s)}{|e^{is}-e^{ix}|^2} \mu'(ds)=
\int_0^{2\pi} \frac{2\sin(x-s)}{2-2\cos(s-x)} \mu'(ds).
\end{equation*}
We obtain
\begin{equation*}
 V_0'(x) = \int_0^{2\pi} \frac{1}{\cos(s-x)-1} \mu'(ds)<0.
\end{equation*}
The theory of the real (inviscid) Burgers equation, see \cite[p. 77, 78]{Mil06}, implies that
 $\exp(i x)$ belongs to $\supp \mu_t$ the first time at $t=T(x)=\frac{-1}{2V_0'(x)}$.
Now put $m=\min_{x\in\R}|V'_0(x)|>0$. Then $T(x)\leq \frac{1}{2m}$ and we see that 
$\supp \mu_t=\T$ the first time for $t=\frac{1}{2m}$.
\end{remark}

\begin{example}
Let $\mu_0=\delta_1$, i.e. $M_0(z)=\frac{1+z}{1-z}$. 
In this case we can choose $\mu'=\delta_0$ and we have 
\begin{equation*}
V_0(x) = \frac{\sin(x)}{1-\cos(x)}, \quad V_0'(x) = \frac{1}{\cos(x)-1}.
\end{equation*}
Thus $m=\frac1{2}$ and $T=\frac1{2m}=1$ is the smallest time with $\supp \mu_T=\T$.
\end{example}

\begin{figure}[h]
\rule{0pt}{0pt}
\centering
\includegraphics[width=6cm]{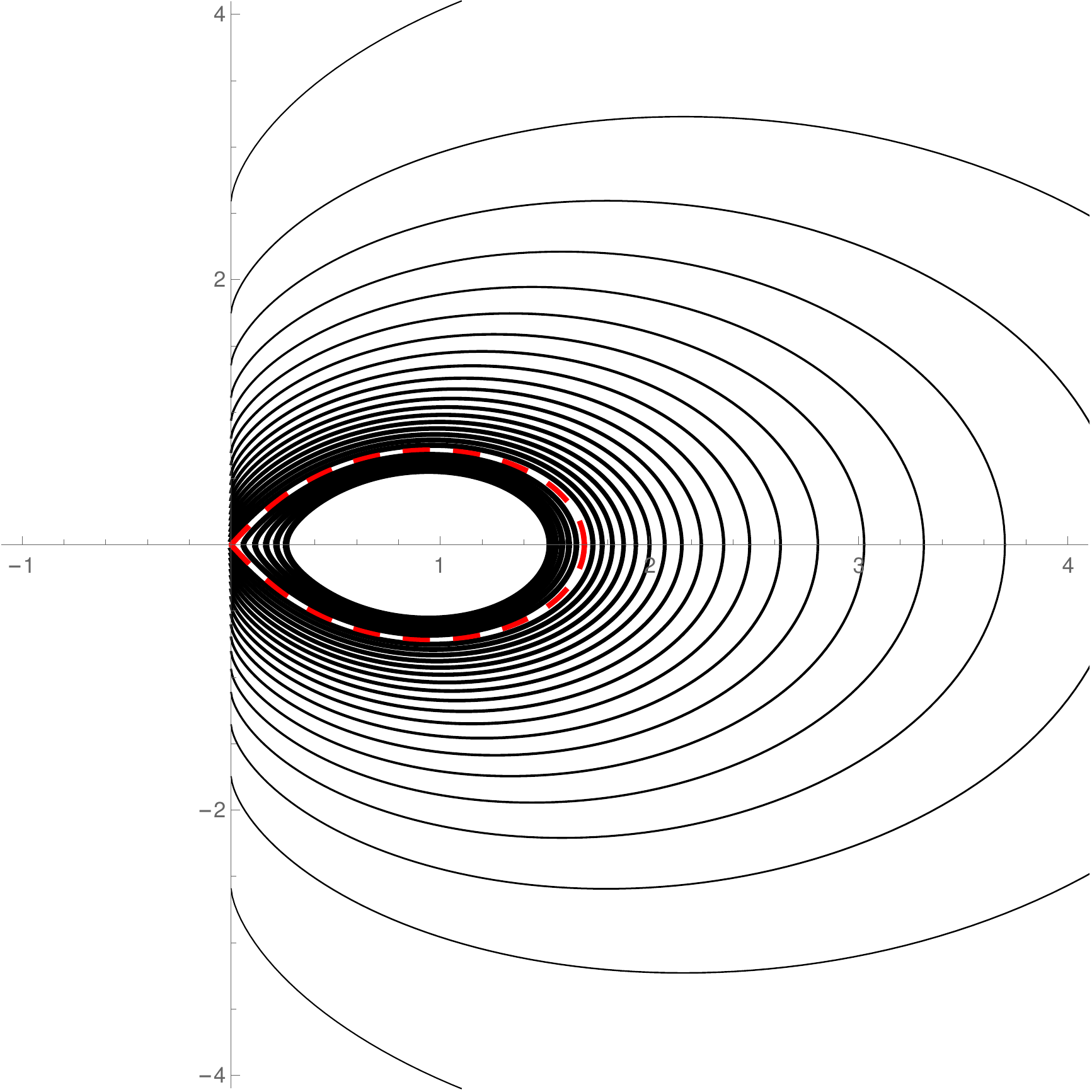}
\hspace{2cm}
\includegraphics[width=6cm]{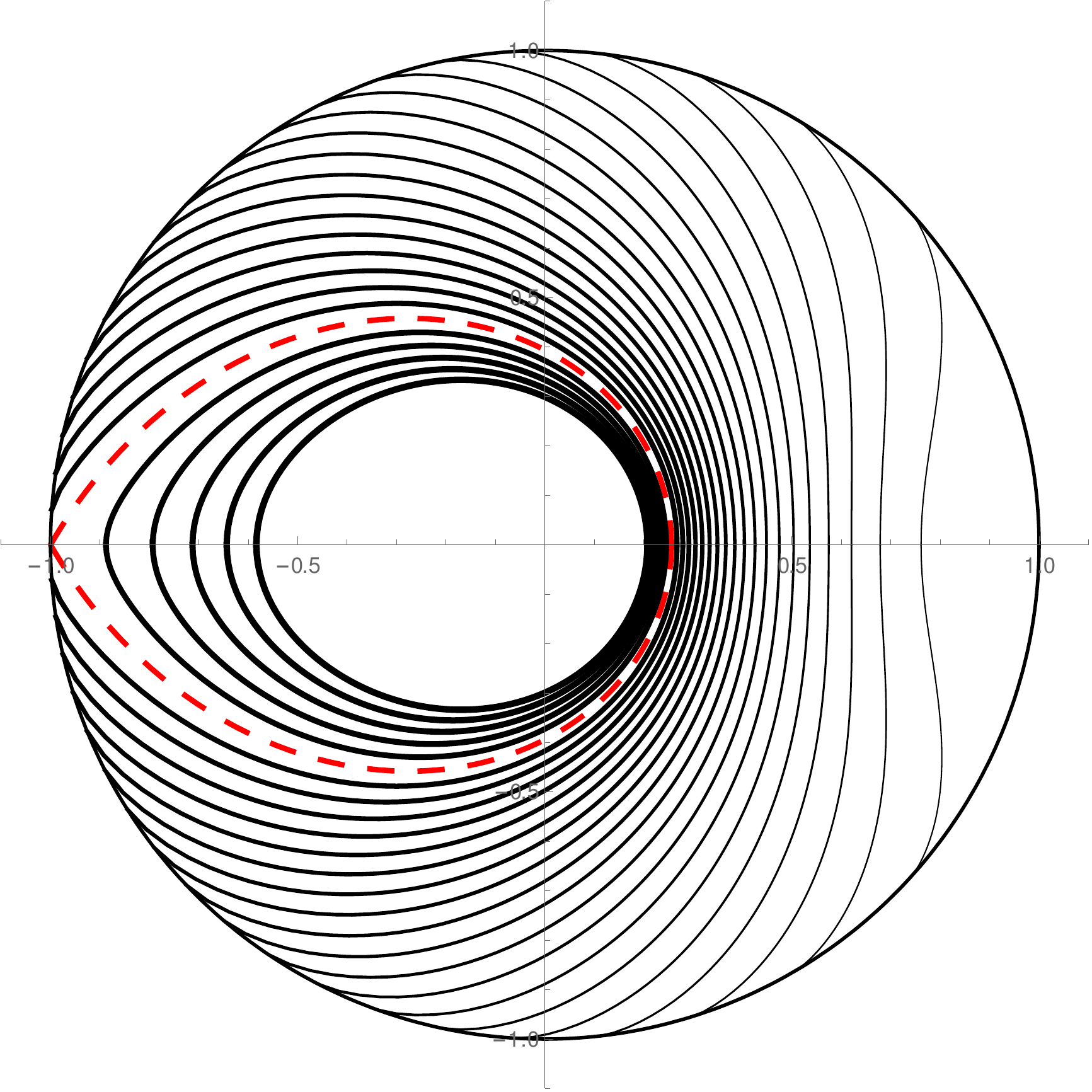}
\caption{Left: The boundary of $M_t(\D)$ for $t=0.05\cdot n$, $n=1,...,25$. 
Right: The mappings $M_0^{-1}\circ M_t$. The {\color{red}red}, dashed curves correspond to $t=1$. }
\label{fig1}
\end{figure}

Finally, we note that the simultaneous case can also be handled by switching the setting to the upper half-plane 
where the Loewner equation becomes quite similar to the one describing the evolution of trajectories of certain quadratic differentials.
Then one can use a result from \cite{MR3764710} to obtain the limit equation in the upper half-plane,
see Appendix \ref{A2}. 

\section{A general Burgers-Loewner equation}\label{sec_burg_loew}

Let $\alpha$ be a probability measure on $\T$ and let $S$ be a holomorphic mapping from the 
right half-plane $RH$ into itself.

The differential equation \eqref{sys_sim} is a special case of the following partial differential equation

\begin{equation}\label{BurLoe} 
\frac{\partial}{\partial t}M_t(z) = -z S(M_t(z)) \cdot \frac{\partial}{\partial z}M_t(z), 
\qquad M_0(z) = \int_\T \frac{x+z}{x-z}\alpha(dx).
\end{equation}

This equation can be interpreted in at least three different ways:

\begin{itemize}
\item[(A)] Each $M_t$ can be written as $M_t(z) = \int_\T \frac{x+z}{x-z}\mu_t(dx)$ for some 
	probability measure $\mu_t$ on $\T$. These measures are related to a semigroup with 
	respect to 
	\emph{free convolution}. \\
	In Section \ref{sec_heat}, we explain this view in more detail. In particular,
	we see that \eqref{BurLoe} always has a unique solution. 
\item[(B)] As in Section \ref{Sec_Sim}, the solution $\{M_t\}$ can be seen as a Herglotz vector field for 
the Loewner equation
	\[\frac{\partial}{\partial t}g_{t}(z) =  g_{t}(z) M_t(g_t(z)),\quad g_{0}(z)=z.\]
	The case $S_\mu(z)=2z$ describes the infinite-slit limit of radial multiple SLE with equal weights. The measures from (A) in this special case describe the distributions of a (time changed) free unitary Brownian motion, see \cite{AWZ14}. (In particular, see \cite[p. 3490]{AWZ14},
        which gives the $\Sigma$-transforms of these measures. The $\Sigma$-transform is defined below.)
\item[(C)] Equation \eqref{BurLoe} can also be regarded as a special case of Loewner's partial differential equation.
In case $M_0$ is univalent, $\{M_t\}_{t\geq0}$ is a decreasing Loewner chain. In general, the solution is 
a decreasing subordination chain. 
	
\end{itemize}

\subsection{Free semigroups}

Let $\mu$ be a probability measure on $\T$. 
The moment generating function is a holomorphic function on $\D$ defined by
\begin{equation*}
\psi_\mu(z)=\int_{\T}\frac{xz}{1-xz}\, \mu({\rm d}x)=\sum_{n=1}^\infty 
\left(\int_\T x^n \, \mu({\rm d}x)\right) z^n, \qquad z \in \D.
\end{equation*}

The classical independence of random variables leads to the classical convolution, 
or Hadamard convolution, $\mu \star \nu$, with $\psi_{\mu \star \nu}=
\sum_{n=1}^\infty 
\left(\int_\T x^n \, \mu({\rm d}x)\right)
\left(\int_\T x^n \, \nu({\rm d}x)\right) z^n$.
Other notions of independence from non-commutative probability theory lead to further 
convolutions. First, we need to define the $\eta$-transform of $\mu$. Let 
\begin{equation*}
\eta_\mu(z)=\frac{\psi_\mu(z)}{1+\psi_\mu(z)}, \qquad z \in \D.
\end{equation*}

\begin{lemma}[See Proposition 3.2 in \cite{BB05}]
Let $\psi\colon\D \to \C$ be holomorphic. The following conditions are
equivalent.
\begin{enumerate}
\item[(1)] There exists a probability measure $\mu$ on $\T$ such that $\psi=\psi_\mu$.
\item[(2)]$\psi(0)=0$ and $\Re(\psi(z)) \geq -\frac{1}{2}$ for all $z\in \D$.
\end{enumerate}
Let $\eta\colon\D \to \C$ be holomorphic. The following conditions are
equivalent.
\begin{enumerate}
\item[(3)] There exists a probability measure $\mu$ on $\T$ such that $\eta=\eta_\mu$.
\item[(4)] $\eta(0)=0$ and $\eta$ maps $\D$ into $\D$.
\end{enumerate}
\end{lemma}

Recall that we denote by $\cP(\T)$ the set of all 
probability measures $\mu$ on $\T$ (Definition \ref{definition3.5}). We let $\cP_{\times}(\T)$ be the set of all $\mu\in \cP(\T)$ with $\eta_\mu'(0)\not=0$,
i.e. the first moment of $\mu$ is $\not=0$. Then 
we can invert $\eta_\mu$ in a neighbourhood of $0$. Denote this locally defined function by 
$\eta_\mu^{-1}$. (For the inversion with respect to multiplication, we will always write $\frac{1}{\eta_\mu}$.)  The $\Sigma$-transform of $\mu$ is defined by 
\begin{equation*}
\Sigma_\mu(z)=\frac{1}{z}\eta_\mu^{-1}(z).
\end{equation*}

For two probability measures $\mu,\nu\in\cP_{\times}(\T)$, Voiculescu \cite{Voi87} characterized
\emph{multiplicative free convolution} $\boxtimes$ by
\begin{equation}\label{FM}
\Sigma_{\mu \boxtimes \nu}(z) = \Sigma_{\mu}(z) \Sigma_{\nu}(z)
\end{equation}
in a neighborhood of $0$. The multiplicative free convolution arises from the notion of free independence 
of unitary operators. An introduction to free probability theory can be found 
in \cite{ns06}.

\begin{remark}
Another notion of independence of unitary operators, the so called 
monotone independence, leads to the 
multiplicative monotone convolution $\mu\rhd \nu$ defined by the composition of 
 the $\eta$-transforms, i.e. $\eta_{\mu\rhd \nu} = \eta_\mu \circ \eta_\nu$.
Monotone independence is closely related to univalent functions. The following two statements are equivalent:
\begin{enumerate}[\rm(1)]
	\item$\eta_\mu$ is univalent on $\D$.
	\item There exists a quantum process $\{X_t\}_{t\geq0}$ of 
	unitary operators with monotonically independent increments such that 
	$\mu$ is the distribution of $X_1$. 
\end{enumerate}
The precise definitions of the ``distribution of an operator'', ``quantum processes'' and ``monotone independence'' can be found in 
\cite{FHS}. In particular, the proof of the above statement is given in 
\cite[Sections 5.1-5.4%and Theorem 3.22...
]{FHS}.
\end{remark}

$\mu\in\cP_{\times}(\T)$ is called \emph{(freely) infinitely divisible} 
if for every $n\in\N$ there exists $\mu_n \in \cP_{\times}(\T)$ such that 
$\mu = \mu_n \boxtimes \cdots \boxtimes \mu_n$ ($n$-fold convolution).
Freely infinitely divisible distributions can be characterized in the following way.

\begin{theorem}[See Lemma 6.6 in \cite{BV92}]\label{MFID}
Let $\mu\in\cP_{\times}(\T)$. The following three statements are equivalent.
\begin{enumerate}[\rm(1)]
\item $\mu$ is freely infinitely divisible.
\item\label{CSUMFI} There exists a weakly continuous $\boxtimes$-convolution semigroup $\{\mu_t \}_{t\geq 0}$ (i.e. $\mu_{t+s}=\mu_t \boxtimes \mu_s$ for all $s,t\geq0$ and $t\mapsto\mu_t$ is continuous with respect to weak convergence) such that $\mu_0 = \delta_1$ and $\mu_1 = \mu$.
\item\label{UMFI} There exists an analytic map $u_\mu\colon\D \to\C$ with $\Re(u_\mu)\geq 0$  such that $\Sigma_\mu(z) = \exp(u_\mu(z))$.
\end{enumerate}
Moreover, the analytic map $u_\mu$ in \eqref{UMFI} can be characterized by the Herglotz representation
\begin{equation}\label{VectorFUMFI}
u_\mu(z) = -i \alpha +\int_{\T} \frac{1+ z \zeta}{1-z\zeta}\rho(d \zeta),
\end{equation}
where $\alpha \in \R$ and $\rho$ is a finite, non-negative measure on $\T$.

Conversely, for any analytic map $u\colon\D \to\C$  with $\Re(u)\geq 0$,  the function $\exp(u(z))$ is the $\Sigma$-transform of some freely infinitely divisible $\mu$.
\end{theorem}

\begin{remark}
 The multiplicative free convolution can also be considered for probability measures $\mu$
 with zero mean. Such a measure is infinitely divisible if and only if $\mu$ is the uniform distribution on 
 $\T$, i.e. $\psi_\mu(z)\equiv 0$, see \cite[Lemma 6.1]{BV92}.
\end{remark}

\subsection{Properties of the Burgers-Loewner equation}\label{sec_heat}

Let $\eta_t=\eta_{\mu_t}$, where $\{\mu_t\}$ is a semigroup as in \eqref{CSUMFI} 
of Theorem \ref{MFID}. Furthermore, assume that $u_\mu$ does not have the form $u_\mu(z)\equiv xi$ for some $x\in \R$. We exclude this simple case which only leads to 
simple rotations of point measures.
Then by \eqref{UMFI} of Theorem \ref{MFID}, we have \[\eta_t^{-1}(z) = z\exp(t u_\mu(z)).\]

This yields the differential equation
\begin{equation}
\frac{\partial}{\partial t}\eta_t(z) = -z u_\mu(\eta_t(z)) \cdot \frac{\partial}{\partial z}\eta_t(z).
\end{equation}

We put $M_t(z)=\frac{1+\eta_t(z)}{1-\eta_t(z)}=1+2\psi_{\mu_t}$. 
Then we obtain the partial differential equation
\begin{equation}\label{addd}
\frac{\partial}{\partial t}M_t(z) = -z S_\mu(M_t(z)) \cdot \frac{\partial}{\partial z}M_t(z), 
\quad M_0(z)\equiv z,
\end{equation}
with $S_\mu(z)=u_{\mu}(\frac{1-z}{1+z})$. 
Note that $S_\mu$ maps the right half-plane $RH$ holomorphically into itself. 
We now consider this equation, the Burgers-Loewner equation, with an arbitrary 
initial value.  

\begin{theorem}\label{at_most}Let $S:RH\to RH$ be holomorphic and $\alpha\in \cP(\T)$.
Then there exists exactly one solution $\{M_t\}_{t\geq0}$ of holomorphic mappings $M_t:\D\to RH$ of 
\begin{equation}\label{eq_heat}\frac{\partial}{\partial t}M_t(z) = -z S(M_t(z)) \cdot \frac{\partial}{\partial z}M_t(z), 
\quad M_0(z) = \int_\T \frac{x+z}{x-z}\, \alpha(dx) = 1+2\psi_\alpha(z).
\end{equation}
There exists a $\boxtimes$-semigroup $(\nu_{t})_{t\geq0} \subset \cP(\T)$ such that 
$M_t = M_0 \circ \eta_{\nu_{t}}$ for all $t\geq0$.
\end{theorem}
In Appendix \ref{A1}, we briefly describe the ``chordal'' or ``additive'' case, 
which corresponds to the free convolution of probability measures on $\R$.
\begin{proof}
For every $t\geq0$, the function $\exp(t S(M_0(z))$ is the $\Sigma$-transform of a 
freely infinitely divisible measure $\nu_{t}$ by Theorem \ref{MFID}. 
Thus $\eta_{\nu_{t}}^{-1}(z) = z\exp(t S(M_0(z))$ and 
\[  \frac{\partial}{\partial t}\eta_{\nu_{t}}^{-1}(z) = \eta_{\nu_{t}}^{-1}(z) S(M_0(z)), 
\quad \eta_{\nu_{0}}^{-1}(z) = z.\]
This yields
\[  \frac{\partial}{\partial t}\eta_{\nu_{t}}(z) =  -z S(M_0(\eta_{\nu_{t}}(z))) \cdot \frac{\partial}{\partial z}\eta_{\nu_{t}}(z), 
\quad \eta_{\nu_{0}}(z) = z.\]
Let $N_{t}=M_0\circ \eta_{\nu_{t}}$. Then 
\[  \frac{\partial}{\partial t}N_{t}(z) =  -z S(N_t(z))) \cdot \frac{\partial}{\partial z}N_t(z), 
\quad N_0=M_0.\]
We see that equation \eqref{eq_heat} has at least one solution.\\

Let $M_t$ be an arbitrary solution and put $\eta_t=\frac{1-M_t}{1+M_t}$. Then 
$\eta_t$ maps $\D$ holomorphically into $\D$ and satisfies
\[ \frac{\partial}{\partial t}\eta_t(z) = -z u(\eta_t(z)) \cdot \frac{\partial}{\partial z}\eta_t(z), \quad
 \eta_0(z) = \eta_\alpha,\]
where $u=S(\frac{1-z}{1+z})$ maps $\D$ into $RH$. Putting $z=0$ yields that $\eta_t(0)=0$ for all $t\geq0$.
We now consider the power series expansions $\eta_t(z)=\sum_{n=1}^\infty a_n(t)z^n$, $\eta_\alpha(z)=\sum_{n=1}^\infty c_nz^n$, 
and $u(z)=\sum_{n=0}^\infty b_nz^n$. \\
The above differential equation yields
\[\dot{a_1}(t)z+\dot{a_2}(t)z^2+... = -z \sum_{n=0}^\infty (b_n(a_1(t)z+a_2(t)z^2+...)^n) \cdot (a_1(t)+2a_2(t)z+3a_3(t)z^2+...)\]
and we obtain a recursive system of differential equations for the coefficients $a_n(t)$, namely
\begin{eqnarray*}
&&\dot{a_1}(t) = -b_0a_1(t), \quad a_1(0) = c_1,\\
&&\dot{a_2}(t) = -2b_0a_2(t)-b_1a_1(t)^2, \quad a_2(0) = c_2,\\
&& etc.
\end{eqnarray*}

We conclude that each function $t\mapsto a_n(t)$ is uniquely determined and thus also $(\eta_t)_{t\geq0}$ and $(M_t)_{t\geq0}$ is uniquely determined. We conclude $M_t=N_t$ for all $t\geq0$.
\end{proof}

By the Herglotz representation formula, every $M_t$ can be written as 
\begin{equation}\label{measures}
M_t(z)=\int_\T \frac{x+z}{x-z}\mu_t(dx) = 1+2\psi_{\mu_t}(z)
\end{equation} for a probability measure $\mu_t$. 

\begin{remark}
We have $M_t = M_0 \circ \eta_{\nu_{t}}$, i.e. $\psi_{\mu_t} = \psi_{\mu_0}\circ \eta_{\nu_t}$. 
In other words: $\mu_t = \mu_0 \rhd \nu_t$.\\
Now let us consider instead $\alpha_t := \mu_0 \boxtimes \nu_t$. 
We have $\eta_{\alpha_t}^{-1}(z) = \frac1{z}\eta_{\mu_0}^{-1}(z) \eta_{\nu_t}^{-1}(z)$ and
locally

\[  \frac{\partial}{\partial t}\eta_{\alpha_{t}}^{-1}(z) = 
  \eta_{\alpha_t}^{-1}(z) S(M_0(z)).\]
This yields
\[  \frac{\partial}{\partial t}\eta_{\alpha_{t}}(z) =  -z S(M_0(\eta_{\alpha_{t}}(z))) 
\cdot \frac{\partial}{\partial z}\eta_{\alpha_{t}}(z).\]

Let $\hat{M}_{t}=1+2\psi_{\alpha_t}(z)=F(\eta_{\alpha_t})$ with $F(z)=\frac{1-z}{1+z}$. (Note that 
$F^{-1}(z)=F(z)$.)
Put $\hat{S}=S\circ M_0\circ F$, which maps $RH$ into $RH$.
Then $\hat{M}_t$ satisfies an equation of the same 
type as \eqref{eq_heat}:
\[  \frac{\partial}{\partial t}\hat{M}_{t}(z) =  -z \hat{S}(\hat{M}_t(z)) \cdot 
\frac{\partial}{\partial z}\hat{M}_t(z), 
\quad \hat{M}_0=M_0.\]

\end{remark}

For $\Im(z)>0$ define $V_t(z)=-iM_t(e^{iz})$. Then $V_t$ maps the upper half-plane into the 
 lower half-plane and we obtain 
\begin{equation}\label{mmm}
 \frac{\partial V_t(z)}{\partial t}= i S(iV_t(z)) \frac{\partial V_t(z)}{\partial z}=
  U(V_t(z)) \frac{\partial V_t(z)}{\partial z}, 
\end{equation}
with $U(z)=iS(iz)$, which maps $-\Ha$ into $\Ha$.

\newpage

\begin{theorem} \label{prob_prop}
Let $\{M_t\}_{t\geq 0}$ be the solution to \eqref{eq_heat}.
\begin{itemize}
\item[(a)] $\{M_t\}_{t\geq 0}$ is a (decreasing and normalized) subordination family. 
 \item[(b)] There exists $T>0$ such that, for all $t\geq T$, $\supp \mu_t = \T$ and 
 $M_t$ can be extended analytically to $\overline{\D}$ with $\Re(M_t(x))>0$ for all $x\in\T$.
Furthermore, $\mu_t$ is absolutely continuous for all $t\geq T$. 
\item[(c)] We have $M_t(z)\to 1$ locally uniformly as $t\to\infty$. Equivalently, 
$\mu_t$ converges weakly to the uniform distribution on $\T$ as $t\to\infty$.
\end{itemize}
\end{theorem}
\begin{remark}The fact that free convolution also leads to subordination was observed
by Voiculescu in \cite[Proposition 4.4]{Voi93}
for the case of compactly supported measures.
The general case was considered in \cite[Theorem 3.1]{Bia98}.
\end{remark}

\begin{proof}

We start with (a). Equation \eqref{eq_heat} is Loewner's partial differential equation 
\eqref{Loe_PDE} with Herglotz vector field $h(z,t)=S(M_t(z))$. We have 
$h(0,t)=S(M_t(0))=S(1)=:a\in RH$. Theorem \ref{Pom_sub}
implies that $\{M_t\}_{t\geq 0}$ is a normalized subordination family. 
This can also be seen directly from the proof of
Theorem \ref{at_most}.\\

Next we prove (b) and we use the function $V_t$ satisfying \eqref{mmm}.
% First assume that $M_0$ is a constant function. Then $M_0(z)\equiv 1$ 
% and thus $M_t(z)=1$ for all $z\in\D$ and $t\geq 0$. In this case, 
% $\mu_t$ is the uniform distribution on $\T$ for all $t\geq 0$.\\
% 
% Now assume that $M_0$ is not constant. Then $V_0$ maps $\Ha$ into the lower half-plane.\\

Let $z_0\in\Ha$ and denote by $z(t)=z(t;z_0)$ the solution to 
\[\dot{z}(t) = -U(V_t(z(t))), \quad z(0)=z_0.\]
A simple calculation shows that $\frac{d}{dt}V_t(z(t))=0$. 
Hence $V_t(z(t))=V_0(z_0)$ for all $t\geq0$ and $\ddot{z}(t)=0$, 
so \[z(t)=z_0-U(V_0(z_0))\cdot t.\] 
Note that $\Im(-U(V_0(z_0)))<0$. So $z(t)$ will hit the real axis 
at some $x(z_0)$ at a certain time $t=t_0(z_0)$. Clearly, $t_0(z_0)>0$ as $z_0\in\Ha$. Furthermore, as $V_t(z)=V_t(z+2\pi)$ for all $z\in\Ha$ and $t\geq 0$, we have $z(t;z_0+2\pi)=z(t;z_0)+2\pi$ and thus 
$t_0(z_0+2\pi) = t_0(z_0)$ for all $z_0\in\Ha$.\\
Next we note that $z(t)$ is defined for all $t\geq0$. This shows 
that $V_{t_0}(z_0)$ can be extended analytically to a neighbourhood 
of $x(z_0)$. 
Now consider $T=\sup_{x\in\R} t_0(i+x)$. Clearly, $T>0$ and we have 
$T<\infty$ since $x\mapsto t_0(i+x)$ is $2\pi$-periodic.\\
Denote the holomorphic function $z_0\mapsto z(t;z_0)$ by $f_t(z_0)$. We have 
$V_t\circ f_t = V_0$. \\
For $t\geq T$, $f_t$ maps a subdomain of $\Im(z)>1$ onto the upper half-plane. 
Hence, $V_t$ can be extended analytically to $\R$ and $\Im(V_t(x))<0$ for all $x\in\R$. 
This shows that $M_t$ can be extended analytically to $\overline{\D}$ and $\Re(M_t(x))>0$ for all 
$x\in\T$ and $t\geq T$.\\

Let $t\geq T$. The analytic extension of $\Re(M_t)$ to the boundary $\T$ and the fact that $\Re(M_t(x))>0$ on $\T$ imply that $\supp \mu_t = \T$. This follows from the inversion formula (see \cite[p.91]{A65})
% \begin{equation*}
% \mu_t(\{p\}) = \frac{1}{2}\lim_{r\uparrow1}(1-r)M_t(rp)=0, \quad p \in \T, \quad \text{and thus} 
% \end{equation*}
\begin{equation*}
\mu_t([e^{i\alpha}, e^{i\beta}]) = \lim_{r \uparrow 1}\int_\alpha^\beta \Re(M_t(r e^{ix}))\, \frac{dx}{2\pi }= \int_\alpha^\beta \Re(M_t(e^{ix}))\, \frac{dx}{2\pi }>0 \quad 
\end{equation*}
for all $0\leq \alpha\leq \beta\leq 2\pi$. Hence $\supp \mu_t = \T$ for all $t\geq T$. Furthermore, the function $[0,2\pi]\ni x\mapsto \mu_t([1,e^{ix}])$ is continuously differentiable and thus absolutely continuous.
Hence $\mu_t$ is absolutely continuous (with respect to the uniform distribution on $\T$) for all $t\geq T$.\\

Finally we show (c). The coefficient function $a_m(t)$ in the proof of Theorem \ref{at_most} 
satisfies a differential equation of the form $\dot{a}_m(t)=-mb_0a_m(t)+f_m(t)$ and 
$a_1(t)=c_1e^{-b_0t}$. We have $b_0=S(1)\in RH$. By induction we see that 
$f_m(t)$ is a linear combination of decreasing exponential functions and thus also 
$a_m(t)$ is a linear combination of decreasing exponential functions and we conclude 
that $\lim_{t\to\infty}a_m(t)=0$ for all $m\geq 1$.\\
This implies that $\frac{1-M_t(z)}{1+M_t(z)}\to 0$ and thus $M_t(z)\to 1$ locally uniformly. Note that the constant function $1$ is the Herglotz function of the uniform distribution, i.e. $1\equiv \int_{0}^{2\pi} \frac{e^{ix}+z}{e^{ix}-z}\frac{dx}{2\pi}$ on $\D$. 
Finally, locally uniform convergence of Herglotz functions is equivalent to weak convergence of the probability measures,  see \cite[Lemma 2.11]{FHS}.
\end{proof}

Recall that $M_t$ has the form $M_t = M_0\circ \eta_{\nu_t}$, where $(\eta_{\nu_t})_{t\geq0}$ 
is a decreasing Loewner chain.\\ 
Due to Theorem \ref{MFID}, each $f_t:=\eta_{\nu_t}^{-1}$ extends analytically to $\D$. 
 So one might wonder whether $(f_t)_{t\geq0}$ is an increasing Loewner chain on $\D$. \\
 This is the case in the simple example $\eta_{\nu_t}(z)=e^{-t}z$, but the 
 following example shows that it is not true in general.
%  The evolution of $f_t$ can be given quite explicitly. Recall from the proof of Theorem 
%  \ref{at_most} that $f_t(z) = z\exp(t S(M_0(z))=z\exp(t S(z))$.  
\begin{example}
Assume that $\eta_\mu$ extends to a conformal mapping from $\C\setminus (-\infty,-1]$ 
onto $\D$ and that $\eta_\mu'(0)>0$. Then $f:=\eta_\mu^{-1}$ is $4$ times the Koebe function, 
$f(z)=\frac{4z}{(1-z)^2}$.\\
Due to Theorem \ref{MFID}, $\mu$ is freely infinitely divisible and 
$u_\mu(z)=\log(f(z)/z)=\log 4 -2\log(1-z)$ maps $\D$ into $RH$.\\
Now let $(\eta_{\nu_t})_{t\geq0}$ be the decreasing Loewner chain from 
Theorem \ref{at_most} with  $S(z)=z$ and $M_0(z)=u_\mu(z)$.
Due to the proof of Theorem \ref{at_most} we have 
\begin{equation}\label{t}
f_t(z):= \eta_{\nu_t}^{-1}(z)=z\exp(t S(M_0(z))=z\exp(t u_\mu(z))
 \end{equation}
and $f_1=f$.\\
Now assume that $(f_t)_{t\geq 0}$ is an increasing Loewner chain. 
Then, for $t>1$, the function $f_t$ maps $\D$ conformally onto a simply connected domain that contains 
$\C\setminus(-\infty,-1]$ with $f_t(0)=0$ and $f'_t(0)=4^t$.
This implies that $f_t(\D)= \C\setminus(-\infty,-4^{t-1}]$ and $f_t = 4^{t-1}f_1$.
But \eqref{t} shows that $t\mapsto f_t/4^{t-1}$ is not constant on $[1,\infty)$. Hence, 
$(f_t)_{t\geq 0}$ is not an increasing Loewner chain. \\
The figure below illustrates $f_t$ and it seems that $(f_t)_{0\leq t\leq 1}$ is indeed 
(a part of) an increasing Loewner chain.
\end{example}

\begin{figure}[ht]
\rule{0pt}{0pt}
\centering
\includegraphics[width=17cm]{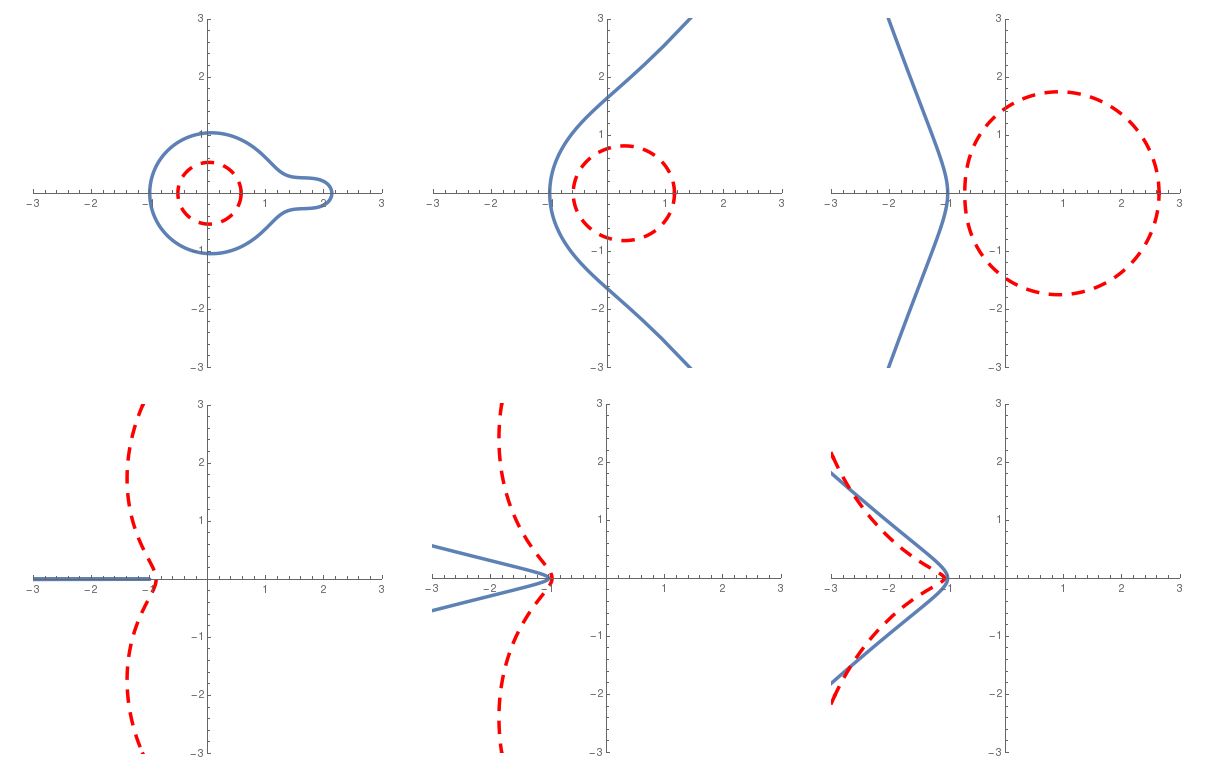}
\caption{The curves $f_t(0.999\T)$ (blue) and $f_t(0.5\T)$ ({\color{red}red}, dashed) for 
$t=0.05, 0.3, 0.6, 1, 1.1, 1.3$.}
\end{figure}

\newpage
\appendix
\section{Appendix}

\subsection{Free additive semigroups}\label{A1}
We briefly describe the ``chordal'' (or ``additive'') analogue of equation \eqref{eq_heat}.\\
 Let $\mu$ be a probability measure on $\R$. The Cauchy transform (or Stieltjes transform) of a probability measure $\mu$ on $\R$ is given by 
$$G_{\mu}(z)=\int_\R\frac{1}{z-t}\,\mu(dt), \qquad z\in\C\setminus \R.$$
The $F$-transform of $\mu$ is simply defined by $F_\mu(z)=1/G_\mu(z)$ for $z\in\Ha=\{z\in\C\,|\, \Im(z)>0\}$.\\
In some truncated cone in $\mathbb{H}$, we can invert $G_\mu$ (with respect to composition). Denote the inverse by $G_\mu^{-1}(z)$. 
The $R$-transform $R_\mu$ is defined by $R_\mu(z)=G_\mu^{-1}(z)-\frac1{z}$. 
Instead of $R_\mu$, one might also regard the Voiculescu transform $\varphi_\mu(z)$ defined by $\varphi_\mu(z)=F_\mu^{-1}(z)-z=G_\mu^{-1}(1/z)-z=R_\mu(1/z)$.\\
 
For two probability measures $\mu$ and $\nu$ on $\R$, the
\emph{additive free convolution} $\boxplus$ is defined by
\begin{equation*}
\varphi_{\mu \boxplus \nu}(z) = \varphi_{\mu}(z)  + \varphi_{\nu}(z).
\end{equation*}
A probability measure $\mu$ on $\R$ is called \emph{(freely) infinitely divisible} 
if for every $n\in\N$ there exists $\mu_n$ such that 
$\mu = \mu_n \boxplus \cdots \boxplus \mu_n$ ($n$-fold convolution).

\begin{theorem}[Theorem 5.10 in \cite{BV93}] \label{thmBV93}
For a probability measure $\mu$ on $\R$, the following statements are equivalent.
\begin{enumerate}[\rm(1)]
\item $\mu$ is infinitely divisible with respect to free convolution $\boxplus$.
\item $\mu=\mu_1$  for a weakly continuous $\boxplus$-semigroup $\{\mu_t\}_{t\geq0}$.
\item For any $t>0$, there exists a probability measure $\mu^{\boxplus t}$ with the property $\varphi_{\mu^{\boxplus t}}(z) = t\varphi_\mu(z).$
\item $R_\mu$ extends to a Pick function, i.e.\ an analytic map of $\mathbb{H}$ into $\mathbb{H} \cup \R$.
\item\label{FLK1} There exist $\gamma \in \R$ and a finite, non-negative measure $\rho$ on $\R$ such that
\begin{equation*}
\varphi_\mu(z)=\gamma +\int_{\mathbb{R}}\frac{1+z x}{z-x} \rho({\rm d}x) ,\qquad z\in \mathbb{H}.
\end{equation*}
The pair $(\gamma,\rho)$ is unique.
\end{enumerate}
Conversely, given $\gamma\in\R$ and a finite, non-negative measure $\rho$ on $\R$, there exists a unique $\boxplus$-infinitely divisible distribution $\mu$ which has the Voiculescu transform of the form \eqref{FLK1}.
\end{theorem}
Let  $\{\mu_t\}_{0\leq t}$ be a $\boxplus$-semigroup and let $G_t = G_{\mu_t}$. Then 
\[ \frac{\partial}{\partial t}G_t(z) = -\frac{\partial}{\partial z}G_t(z) \cdot R_\mu(G_t(z)), \quad G_0(z)=1/z.\]
This is a special case of Loewner's partial differential equation on the upper half-plane $\mathbb{H}$ due to property (4) of the above theorem. \\
Put $F_t=F_{\mu_t}=1/G_t$. Then
\[ \frac{\partial}{\partial t}F_t(z) = -\frac{\partial}{\partial z}F_t(z) \cdot \varphi_\mu(F_t(z)), \quad F_0(z)=z.\]
This equation corresponds to \eqref{addd} and we obtain the analogue of \eqref{eq_heat} by regarding an arbitrary initial value 
$F_0 = F_\alpha$ for some probability measure $\alpha$.

\subsection{Radial multiple SLE in the upper half-plane}\label{A2}

By Remark \ref{chordal_version}, we can regard radial multiple SLE 
also in $\Ha$ by using the Cayley transform $C$ (provided that $x_{N,k}$ is not $1$) and switching to the chordal
multi-slit Loewner equation
 \begin{eqnarray*}
dg_{N,t}(z) = \sum_{k=1}^N \frac{2\lambda_{N,k}}{g_{N,t}(z)-U_{N,k}(t)}dt, \quad g_{N,0}(z)=z,
\end{eqnarray*}
with 
\begin{eqnarray*}
&&dU_{N,k}(t) = \sqrt{\kappa\lambda_{N,k}}dB_t + \sum_{j\not= k} \frac{2 (\lambda_{N,k}+\lambda_{N,j})}{U_{N,k}(t)-U_{N,j}(t)}dt + \lambda_{N,k}(\kappa-4-2N) \Re\left(\frac{1}{U_{N,k}(t)-S_N(t)}\right)dt,\\
&& U_{N,k}(0)=C(x_{N,k})\in\R,\\
&& dS_N(t) = \sum_{k=1}^N \frac{2\lambda_{N,k}}{S_N(t)-U_{N,k}(t)}dt, \quad S_N(0) = i = C(0).
\end{eqnarray*}
If we drop the second $dt$-term for $U_{N,k}$, we obtain chordal multiple SLE. Questions concerning the limit $N\to\infty$ for this chordal case are discussed in \cite{MR3764710}.\\
However, we can also use \cite{MR3764710} to obtain a statement about the limit 
equation for the radial multiple SLE mappings if $\lambda_{N,k}=1/N$ for all $k$ and $N$. In this case, the evolution of the curves is quite similar to the evolution of
trajectories of a certain quadratic differential and the proof  of \cite[Theorem 3.3]{MR3764710} can be easily adjusted to get the following result.\\

In contrast to Theorem \ref{sim_thm}, however, we only obtain a convergent subsequence and we need an additional condition 
on the starting points of the SLE curves.

\begin{theorem} Let  $\lambda_{N,k}=1/N$ and $y_{N,k}=C(x_{N,k})$.  Assume that there exists $M>0$ such that 
$y_{N,k}\in [-M,M]$ for all $N$ and $k$. Finally, assume that \[  
\sum_{k=1}^N \frac1{N}\delta_{y_{N,k}} \to \mu\]
with respect to weak convergence, where $\mu$ is a probability measure on $\R$. Then there exists a $T>0$ and a subsequence $\{g_{N_k,t}\}_k$ which converges for every $t\in[0,T]$ in distribution with respect to locally uniform convergence and the limit process $g_t$ satisfies
 \begin{eqnarray*}
&&\frac{\partial }{\partial t}g_t(z)=M_t(g_t), \quad g_0(z)=z,\\
&& \frac{\partial }{\partial t}M_t(z) = -2\frac{\partial }{\partial z}M_t(z) \cdot M_t(z) - 2\Re\left(\frac{M_t(z)}{(z-S(t))^2}-\frac{M_t(S(t))}{(z-S(t))^2}-\frac{\frac{\partial }{\partial z} M_t(z)}{z-S(t)}\right), \quad g_{0}(z)=z,\\
&&\text{$S(t)=g_t(i)$, i.e. $\frac{d}{d t}S(t)=M_t(S(t)), \quad 
S(0)=i$.}
\end{eqnarray*}

\end{theorem}
\begin{proof}
By comparing the SDE for $U_{N,k}$ with the one for $V_{N,k}$ in \cite[Section 3]{MR3764710}, we see that our setting corresponds to $M_N=1$, $\alpha_{N,1}=(\kappa-4-2N)/2$ and $s_{N,1}=i$, i.e. $\sigma_{N,0}=
\frac{\kappa-4-2N}{2N}\delta_i\to -\delta_i=\sigma$. Furthermore, 
$2 (\lambda_{N,k}+\lambda_{N,j})=4/N$ leads to an additional factor $2$ 
in front of $\frac{\partial }{\partial z}M_t(z) \cdot M_t(z)$.
The proof is now analogous to the proof of \cite[Theorem 3.3]{MR3764710}.
\end{proof}

%\bibliographystyle{amsalpha}
%\bibliography{bibdata} 
\def\cprime{$'$}
\providecommand{\bysame}{\leavevmode\hbox to3em{\hrulefill}\thinspace}
\providecommand{\MR}{\relax\ifhmode\unskip\space\fi MR }
\providecommand{\MRhref}[2]{%
  \href{http://www.ams.org/mathscinet-getitem?mr=#1}{#2}
}
\providecommand{\href}[2]{#2}

\end{document}